\documentclass[10pt]{amsart}
\usepackage[margin=1.4in]{geometry}
\usepackage{colortbl}
\numberwithin{equation}{section}

\newtheorem{theorem}{Theorem}[section]
\newtheorem{lemma}{Lemma}[section]

\newtheorem{proposition}{Proposition}[section]
\newtheorem{remark}{Remark}[section]
\newtheorem{definition}{Definition}[section]

\newcommand{\R}{\mathbb R}

\newcommand{\N}{\mathbb N}

\newcommand{\bq}{\begin{equation}}
\newcommand{\eq}{\end{equation}}

\newcommand{\e}{\varepsilon}
\newcommand{\lt}{\left}
\newcommand{\rt}{\right}
\newcommand{\mc}{\mathcal{C}}

\newcommand{\pa}{\partial}

\newcommand{\intr}{\int_{\R_+}}
\newcommand{\intrr}{\int_{\R \times \R_+}}

\newcommand{\mt}{\mathcal{T}}
\newcommand{\mh}{\mathcal{H}}

\begin{document}

\title[Paveri-Fontana model]{Existence and hydrodynamic limit for a Paveri-Fontana type kinetic traffic model}

%\author{}
%\address{}
%    Address of record for the research reported here
%\address{Division of Applied Mathematics, Brown University, Providence, RI, USA.}
%    Current address
%\curraddr{Department of Mathematics and Statistics,Case Western Reserve University, Cleveland, Ohio 43403}
%\email{parksajune@skku.edu, }
%    Information for first author
\author[Choi]{Young-Pil Choi}
\address[Young-Pil Choi]{\newline Department of Mathematics\newline
Yonsei University, 50 Yonsei-Ro, Seodaemun-Gu, Seoul 03722, Republic of Korea}
\email{ypchoi@yonsei.ac.kr}

\author[Yun]{Seok-Bae Yun}
%    Address of record for the research reported here
\address[Seok-Bae Yun]{\newline Department of Mathematics\newline Sungkyunkwan University, Suwon 440-746, Republic of Korea}
%    Current address
%\curraddr{Department of Mathematics and Statistics,Case Western Reserve University, Cleveland, Ohio 43403}
\email{sbyun01@skku.edu}
%\homepage{http:\backslash\backslash www.worldpress.com$\backslash$ seokbaeyun}
%    \thanks will become a 1st page footnote.
%\thanks{The first author was supported in part by NSF Grant \#000000.}
%    Information for first author

%    General info
%\subjclass[2000]{Primary 54C40, 14E20; Secondary 46E25, 20C20}

%\date{January 1, 2001 and, in revised form, June 22, 2001.}
%\dedicatory{This paper is dedicated to our advisors.}


\begin{abstract} 
We study a Paveri-Fontana type model, which describes the evolution of the mesoscopic distribution of vehicles through a combined effect of adjustment of the velocity with respect to nearby vehicles, and slowing down and speeding up of the vehicles arising as a result of exchange of velocity with the vehicles on the same location on the road. We first prove the global-in-time existence of weak solutions. The proof is via energy, $L^p$, and compact support estimates together with velocity averaging lemma. The combined effect of alignment nature of $Q_r$, which keeps the characteristic from spreading, and the dissipative nature of $Q_i$, which gives the uniform control on the size of the distribution function,
is crucially used in the estimates. We also rigorously establish a hydrodynamic limit to the presureless Euler equation by employing the relative entropy combined with the Monge-Kantorovich-Rubinstein distance.
% we show that the weak solution of the kinetic equation converges to the strong solution of pressureless Euler equations.
\end{abstract}

\maketitle \centerline{\date}
\tableofcontents

\section{Introduction}
The study of the traffic flow at the kinetic level started in the seminal work \cite{Prigogine1,Prigogine2}
in which Prigogine proposed a kinetic model that explains the traffic flow in a single lane
road through the transport,  exchange of velocity, and the relaxation to the
desired velocity distribution. The relaxation part, however, was criticized by many to be unrealistic, since the desired velocity distribution is a priori prescribed. In order to remedy this drawback, Paveri-Fontana \cite{P-F} proposed a model where the density function has a further state variable, desired velocity, see Section \ref{ssec_PF} for more details. In the case when the desired velocity is fixed to be the local vehicle velocity, which is reasonable since the drivers in the road adjust their velocities according to the vehicles around it, the Paveri-Fontana model takes the following form:
\bq\label{main_sys}
\pa_t f + v \pa_x f = Q(f),
\eq
subject to initial data
\[
f(x,v,0) =: f_0(x,v). %\quad (x,v) \in \R \times \R_+
\]
The vehicle distribution function $f=f(x,v,t)$ denotes the number density function on the phase point $(x,v) \in \R \times \R_+$ at time $t\in \R_+$. The local vehicle density $\rho = \rho(x,t)$ and the local vehicle velocity $u = u(x,t)$ are defined by
\[
\rho(x,t) := \int_0^\infty f(x,v,t)\,dv \quad \mbox{and} \quad u(x,t) :=  \frac{\int_0^\infty vf(x,v,t)\,dv}{\int_0^\infty f(x,v,t)\,dv},
\]
respectively. The operator $Q = Q(f)$, referred to very often as the collision operator in the literature of collisional kinetic theory, is in charge of the interactions between vehicles and their effects on the states. The operator $Q$ consists of relaxation operator $Q_r = Q_r(f)$ and the interaction operator $Q_i = Q_i(f)$:
$$
Q(f)=Q_r(f)+Q_i(f).
$$ 
In our case, the relaxation operator $Q_r$ is given by 
$$Q_r(f) = \pa_v ((v-u)f),$$
which explains the driver's adjustment of velocity with respect to the traffic condition around it. The interaction operator $Q_i$ is presented as the difference of the gain term $Q_i^+ = Q_i^+(f)$ and the loss term $Q_i^- = Q_i^-(f)$: 
\[
Q_i(f)=Q_i^+(f)-Q_i^-(f),
\]
where
\[
Q_i^+(f)=f(x,v,t)\int^{\infty}_v(v_*-v)f(x,v_*,t)\,dv_*
\]
and
\[
Q_i^-(f)=f(x,v,t)\int^{v}_0(v-v_*)f(x,v_*,t)\,dv_*.
\]
The gain term $Q_i^+$ represents the slowing down of the cars running faster than $v$, while the loss term $Q_i^-$ denotes the acceleration of the cars running slower than $v$.  Putting together, the operator $Q$ leads to the concentration of the velocity distribution. In particular, the interaction operator $Q_i$ can be more clearly manifested in the following reformulation:
\[
Q_i(f)(x,v,t)=f(x,v,t)\int^{\infty}_0(v_*-v)f(x,v_*,t)\,dv_*=\rho(x,t) \big(u(x,t)-v\big)f(x,v,t).
\]
In view of this, we can understand that the relaxation term is attracting the trajectory towards the desired velocity, while the collision operator rearranges the velocity distribution of the cars to be concentrated on the desired velocity. Therefore, our model \eqref{main_sys} indicates that the car distribution will eventually converge to the fluid dynamic mono-kinetic configuration.

Being one of the pioneering models in the kinetic theory of traffic flow, to the best of authors knowledge, the existence of solutions for (\ref{main_sys}) and appropriate scaling limit have never been studied in the literature, which is the main motivation of the current work.

\subsection{Formal derivation of \eqref{main_sys}\label{ssec_PF} from the Paveri-Fontana model} As mentioned before, Paveri-Fontana model takes into account the desired velocity as a further state variable $w$. More precisely, the Paveri-Fontana model \cite{P-F} is given by
\begin{align}\label{PF}
\begin{aligned}
\pa_t g + v\pa_x g &= -\pa_v((w-v)g/T) + (1-P)f(x,v,t)\int_v^\infty (v_* - v)g(x,v_*,w,t)\,dv_* \cr
&\quad - (1-P)g(x,v,w,t)\int_0^v (v-v_*)f(x,v_*,t)\,dv_*,
\end{aligned}
\end{align}
where $g = g(x,v,w,t)$ is a generalized distribution, $T$ and $P$ denote the relaxation time and the probability of passing, both depend on the $g$, respectively. Note that the density function $f$ can be recovered by integrating $g$ with respect to the desired velocity variable $w$, i.e.,
\[
\int_{\R_+} g(x,v,w,t)\,dw = f(x,v,t).
\]
This together with \eqref{PF} yields that 
\[
\pa_t f + v \pa_x f = - \pa_v \lt(\frac1T \int_{\R_+} wg\,dw - \frac{vf}{T}\rt) + (1-P)\rho(u-v)f.
\]
We now assume that the desired velocity is given as the local vehicle velocity, i.e.,
\[
g(x,v,w,t) = f(x,v,t)\otimes \delta_{u(x,t)}(w).
\]
Furthermore, we assume that the relaxation time $T$ and the probability of non-passing $1-P$ are constants and normalized to unity. These simplifications lead to our main kinetic traffic flow model \eqref{main_sys}.

\subsection{Main results} Before we define our solution concept and state the main results, we introduce several norms, function spaces, and notational conventions.
\begin{itemize}
\item For functions $f(x,v)$ and $g(x)$, $\|f\|_{L^p}$ and $\|g\|_{L^p}$ denote the usual $L^p(\R \times \R_+)$-norm and $L^p(\R)$-norm, respectively.
\item For any nonnegative integer $s$, $H^s$ denotes the $s$-th order $L^2$ Sobolev space.
\item $\mc^s([0,T];E)$ is the set of $s$-times continuously differentiable functions from an interval $[0,T]\subset \R$ into a Banach space $E$, and $L^p(0,T;E)$ is the set of the $L^p$ functions from an interval $(0,T)$ to a Banach space $E$.
\item We  denote by $C$ a generic, not necessarily identical, positive constant. $C = C(\alpha, \beta, \dots )$ or $C = C_{\alpha,\beta,\dots}$ represents the positive constant depending on $\alpha, \beta, \dots$.
\end{itemize}
We then define the notion of weak solutions to the system \eqref{main_sys}.
\begin{definition}\label{def_weak} Let $T>0$. We say that $f$ is a weak solution to the system \eqref{main_sys} on the time interval $[0,T]$ if the following conditions are satisfied:
\begin{itemize}
\item[(i)] $f \in L^\infty(0,T;(L^1_+ \cap L^\infty)(\R \times \R_+))$,
\item[(ii)] for all $\phi \in \mc^1_c(\R \times \R_+ \times [0,T])$ with $\phi(\cdot,\cdot,T) = 0$,
\begin{align*}
&-\int_{\R \times \R_+} f_0(x,v)\phi(x,v,0)\,dxdv - \int_0^T\int_{\R \times \R_+} f\lt( \pa_t \phi + v\pa_x \phi + (u -v)\pa_v \phi\rt)dxdvdt\cr
&\hspace{3cm}= \int_0^T \int_{\R \times \R_+}  \rho(u-v)f \phi\,dxdvdt.
\end{align*}
\end{itemize}
\end{definition}
We are now ready to state our first main result on the existence theory. 
\begin{theorem}\label{thm_main} Let $T >0$. Suppose that the initial data $f_0 \in L^\infty(\R \times \R_+)$ is compactly supported in both position and velocity. We also assume that $f_0(x,0)=0$. Then there exists at least one weak solution to the equation \eqref{main_sys} in the sense of Definition \ref{def_weak} such that
\begin{itemize}
\item[(i)] $f$ is compactly supported both in position and velocity,  with $f(x,0,t)=0$,
\item[(ii)] $\|f\|_{L^\infty(\R \times \R_+ \times (0,T))} \leq C\|f_0\|_{L^\infty(\R \times \R_+)}$,
\item[(iii)] The total energy is non-increasing in time:
\[
\frac12\int_{\R \times \R_+} v^2 f\,dxdv + \int_0^t \int_{\R \times \R_+} (u -v)^2f\,dxdvds \leq \frac12\int_{\R \times \R_+} v^2 f_0\,dxdv,
\]
for almost every $t \in (0,T)$.
\end{itemize}
\end{theorem}
\begin{remark}
The statement (i) in the theorem above means that the vehicles can run only at a finite speed, and never come to a halt on the road. The former is natural in that the vehicles cannot run with an infinite speed, and the latter is reasonable when, for example, the vehicles are running on a highway. 
\end{remark}
\begin{remark}
The property $f(x,0,t)=0$ is also crucially used in the proof since otherwise the boundary terms in the integration by parts with respect to the velocity variable do not vanish.
%(2) that the condition $f_0(x,0)>0$ states that vehicles cannot halt in the road. Which is relavent, for exmpale, on highways. It wil be verifed using the method of characteristic that the vehecles never stop on the road if every car was running initially.
\end{remark}
The proof combines  the alignment effect of $Q_r$ and the dissipative nature of $Q_i$. 
We first construct approximate solutions $f_\e$ parametrized by a smoothing parameter $\e$, see Section 3.3 for details. 
The alignment property of $Q_r$ prevents the spreading of the characteristics and keep the support of the distriubtion function compact, leading to the conclusion that the approximate solution $f_\e$ is bounded for each $\e$ globally in time.
Then, the dissipative nature of $Q_i$ with respect to $vdv$ and $v^2dv$  gives the coercivity that controls the boundedness of such approximate solutions uniformly in $\e$:
\begin{align*}
	\|f_\e(\cdot,\cdot,t)\|_{L^\infty} &\leq C\|f_{0,\e}\|_{L^\infty} + C\int_0^t \int_{\R \times \R_+} \frac{\rho_\e(u_\e - v)^2 f_\e}{1+ \e \rho_\e (1 + u_\e)}\,dxdvds.
	%&\leq C\|f_{0,\e}\|_{L^\infty} + C_T\int_{\R \times \R_+} (1 + v^2)f_{0,\e}\,dxdv,
\end{align*}
These two estimates, the boundedness of $f_\e$ and the non-spreading of trajectories, are strong enough enough to control the high nonlinearity of the collision operator in the weak limit, and enables one to derive the global-in-time weak solution, instead of having to resort to weaker notions of solutions such as the renormalized solutions.

Our next goal is to study the asymptotic analysis of \eqref{main_sys}. More precisely, we are interested in the limit $\e \to 0$ in the following equation:
\begin{align}\label{main_hydro}
	\begin{aligned}
&\pa_t f^\e + v \pa_x f^\e =  \frac1\e Q(f^\e) =  \frac1\e\pa_v ((v-u^\e)f^\e) + \frac1\e \rho^\e(u^\e - v)f^\e,\cr
\end{aligned}
\end{align}
subject to the initial data 
\[
f^{\varepsilon}(x,v,0)=:f^{\varepsilon}_{0}(x,v)
\]
with
\begin{align*}
	\rho^\e(x,t) &:= \int_0^\infty f^\e(x,v,t)\,dv, \quad\rho^\e(x,t) u^\e(x,t) := \int_0^\infty vf^\e(x,v,t)\,dv,
\end{align*}
which is obtained from rewriting  \eqref{main_sys} in the typical Euler scaling:
\[
x \rightarrow x/\e\quad\mbox{and} \quad t \rightarrow t/\e.
\]
As $\e \to 0$, we expect the mono-kinetic ansatz for $f^\e$, i.e., 
\[
f^\e(x,v,t) \to \rho(x,t) \otimes \delta_{u(x,t)}(v)
\]
in the sense of distributions. Here $\rho$ and $u$ satisfy the following pressureless Euler equations:
\begin{align}\label{Euler_eq}
	\begin{aligned}
	&\pa_t \rho + \pa_x (\rho u) = 0,\cr
	&\pa_t (\rho u) + \pa_x (\rho u^2) = 0.
\end{aligned}
\end{align}

The main tool we employ for the hydrodynamic limit from \eqref{main_hydro} to \eqref{Euler_eq} is based on the relative entropy method for kinetic flocking model developed in \cite{KMT15}, see also \cite{Dafe}. In \cite{KMT15}, however, the diffusion term plays a significant role in that the entropy functional is strictly convex with respect to both $\rho$ and $\rho u$. In the absence of the diffusion term, the functional is not convex for $\rho$ any more, see Section \ref{sec_hydro} for details, and the strategy used in \cite{KMT15} breaks down. Remedies to overcome this are suggested in recent works \cite{CC_pre, C_pre, FK19}, in which a suitable estimate with respect to the second-order Wasserstein distance is augmented to provide the convergence of $\rho^\e$. Inspired by these works, we adopt the Monge-Kantorovich-Rubinstein (in short MKR) distance:
\[
d_{MKR}(\rho_1,\rho_2) := \inf_\gamma \int_{\R \times \R} | x- y| \,d\gamma(x,y), \quad \mbox{for} \quad \rho_1, \rho_2 \in \mathcal{P}_1(\R),
\]
where the infimum runs over all transference plans, i.e., all probability measures $\gamma$ on $\R \times \R$ with first and second marginals $\rho_1$ and $\rho_2$, respectively. Here $\mathcal{P}_1(\R)$ stands for the set of probability measures on $\R$ with finite first-order moment. Note that MKR distance is equivalent to the bounded Lipschitz distance, and it is also called first-order Wasserstein distance. We employ the idea recently developed in \cite{C_pre} to have the quantitative estimate for error between local densities in MKR distance under suitable assumptions on the initial data.

In order to state our second main result of the present work, we present the notion of strong solutions to the pressureless Euler system \eqref{Euler_eq} in the proposition below. 
\begin{proposition}Let $s>2$, and suppose $(\rho_0,u_0) \in H^s(\R) \times H^{s+1}(\R)$ and $\rho_0 > 0$ on $\R$. Then, for any positive constants $N<M$, there exists a positive constant $T_0$ depending only on $N$ and $M$ such that if $\|(\rho_0, u_0)\|_{H^s \times H^{s+1}} \leq N$, then the system \eqref{Euler_eq} with $(\rho_0, u_0)$ admits a unique solution $(\rho,u) \in \mc([0,T_0];H^s(\R)) \times \mc([0,T_0];H^{s+1}(\R))$ satisfying
\[
\sup_{0 \leq t \leq T_0} \|(\rho(\cdot,t), u(\cdot,t))\|_{H^s \times H^{s+1}} \leq M.
\]
\end{proposition}
\begin{remark}
The weak solution for the scaled Paveri-Fontana model (\ref{main_hydro}) can be shown to exist globally in time using the exactly same argument as was employed for Theorem \ref{thm_main}. The existence and uniqueness of strong solutions for the pressureless Euler system \eqref{Euler_eq} can be obtained by using almost the same argument as in \cite{CK16}.
\end{remark}

\begin{theorem}\label{thm_hydro} Let $f^\e$ be a weak solution to the equation \eqref{main_hydro} and $(\rho,u)$ be a strong solution to the system \eqref{Euler_eq} on the time interval $[0,T]$. Suppose that the initial data $f^\e_0$ and $(\rho_0, u_0)$ satisfy the following assumptions.
	\begin{itemize}
		\item[{\bf (H1)}] The initial data are well-prepared:
		\[
		\int_\R \rho_0^\e(u_0 - u^\e_0)^2 \,dx = \mathcal{O}(\e) \quad \mbox{and} \quad \int_\R\lt( \intr v^2 f_0^\e \,dv  - \rho_0 u_0^2\rt)\,dx  = \mathcal{O}(\e).
		\]
		\item[{\bf (H2)}] The local densities $\rho_0$ and $\rho^\e_0$ satisfy
		\[
		d_{MKR} (\rho_0, \rho^\e_0) = \mathcal{O}(\sqrt\e).
		\]
	\end{itemize}
	Then we obtain
	\[
	\sup_{0 \leq t \leq T} \int_\R \rho^\e(x,t)(u^\e-u)^2(x,t)\,dx \leq \mathcal{O}(\e) \quad \mbox{and} \quad \sup_{0 \leq t \leq T}d_{MKR}(\rho^\e(\cdot,t), \rho(\cdot,t)) \leq \mathcal{O}(\sqrt\e).
	\]
	In particular, we have
	\[
	\sup_{0 \leq t \leq T}d_{MKR}(f^\e(\cdot,\cdot,t), \rho(\cdot,t)\otimes \delta_{u(\cdot,t)}(\cdot)) \leq \mathcal{O}(\sqrt\e),
	\]
	that is, $f^\e$ converges to $\rho\otimes \delta_u$ weakly-* as measures. 
	
\end{theorem}

We now briefly overview references on the traffic models relevant to our works. We confine ourselves to the kinetic traffic models, since the literature is enormous, see \cite{B-C,K-W3,PT11} for a survey on mathematical models of vehicular traffic at different scales of descriptions. Prigogine in \cite{Prigogine1,Prigogine2} suggested a Boltzmann type traffic model, which is, as mentioned above, the first kinetic model for traffic flow. The Paveri-Fontana model, which is the main concern of the current work, is introduced in \cite{P-F} to remedy the controversies on the assumption used in the Prigogine model that the traffic system relaxes to a fixed velocity configuration. Klar and Wigner discussed the necessity of considering the length of the interaction, and introduced Enskog type kinetic models in \cite{K-W1,K-W2}. Vlasov-Fokker-Planck type models can be found in \cite{H-P,I-K-M}. In \cite{P-S-T-V}, Puppo et al introduced a Boltzmann type traffic model for which an analytic expression for the steady states can be secured. The BGK type relaxational approximate model for traffic is derived in \cite{H-P-R-V}. For kinetic equations for multi-lane traffic model, we refer to \cite{H-G,I-K-M}. We also note that the studies of the closely related velocity alignment type kinetic equations have attracted a lot of attentions recently since they arise in various contexts such as a coarse grain limit of relevant flocking models \cite{CCP17, C16, CHL17, C-S1,KMT13,KMT15,M-T}.\newline

This paper is organized as follows. In Section \ref{sec_pre}, we provide several preliminary lemmas. We introduce a regularized equation of \eqref{main_sys} mollifying the local vehicle velocity $u$ in the relaxation term $Q_r$ and the interaction operator $Q_i$ in Section \ref{sec_gl_reg}. We also provide some uniform bound estimates of the approximated solutions with respect to the regularization parameters. Using these uniform bound estimates, we prove the global-in-time existence of weak solutions in Section \ref{sec_gl_weak}. Finally, in Section \ref{sec_hydro}, using the relative entropy method combined with the estimate of MKR distance between local densities, we show that the scaled Paveri-Fontana model \eqref{main_hydro} converges to the pressureless Euler system \eqref{Euler_eq}. 
%%%%%%%%%%%%%%%%%%%%%%%%%%%%%%%%%%%%%%%%%%%%%%%%%%%%%%%%%%%%%%%%%%%%%%%%%%%%%%%%%%%%%%
%
%  \section{Preliminaries}
%
%%%%%%%%%%%%%%%%%%%%%%%%%%%%%%%%%%%%%%%%%%%%%%%%%%%%%%%%%%%%%%%%%%%%%%%%%%%%%%%%%%%%%%

\section{Preliminaries: Auxiliary Lemmas}\label{sec_pre}

In this section, we provide several auxiliary lemmas which will be used significantly for our results later. We first state the velocity averaging lemma whose proof can be found in \cite{CCK14, G-L-P-S, KMT132}.
\begin{lemma}\label{lem_veloa}
For $1 \leq p < 3/2$, let $\{g^n\}_{n \in \N}$ be bounded in $L^p(\R \times \R_+ \times (0,T))$. Suppose that $f^n$ is bounded in $L^\infty(0,T;(L^1 \cap L^\infty)(\R \times \R_+))$ and $v^2 f^n$ is bounded in $L^\infty(0,T;L^1(\R \times \R_+))$. If $f^n$ and $g^n$ satisfy the equation
\[
\partial_tf^n + v\pa_x f^n = \pa_v^k g^n, \quad f^n|_{t=0} = f_0 \in L^p(\R \times \R_+),
\]
for a multi-index $k$. Then, for any $\psi(v)$, such that $|\psi(v)| \leq Cv$ as $v \to +\infty$, the sequence
\[
\lt\{ \int_{\R_+} f^n \psi(v)\,dv\rt\}_{n \in \N}
\]
is relatively compact in $L^p(\R \times (0,T))$.
\end{lemma}
In the lemma below, we show the dissipative estimates of the interaction operator $Q_i(f)$.
\begin{lemma}\label{lem_Q}Suppose that $f$ is a solution of \eqref{main_sys} with sufficient integrability. Then we have
	\begin{align*}
	&\int^{\infty}_0Q_i(f)\,dv=0,\quad\int^{\infty}_0vQ_i(f)\,dv=-\int^{\infty}_0\rho(u-v)^2f\,dv,\quad
	\int^{\infty}_0v^2Q_i(f)\,dv\leq0.
	\end{align*}
\end{lemma}
\begin{proof} By definition of $u$, we easily obtain
\[
\int^{\infty}_0Q_i(f)\,dv=\rho\int^{\infty}_0 (u-v)f\,dv = 0.
\]
Using this fact, we compute
	\begin{align*}
	\int^{\infty}_0vQ_i(f)\,dv 
	=\int^{\infty}_0\int_{\R} \rho v(u - v) f\,dxdv
	= -\int^{\infty}_0\int_{\R} \rho(u - v)^2 f\,dxdv.
	\end{align*}
	In order to prove the third relation, we observe that
	\[
	\int^{\infty}_0 v^2 f\,dv \leq \lt(\int^{\infty}_0 f\,dv\rt)^{1/3} \lt(\int^{\infty}_0 v^3 f\,dv \rt)^{2/3} 
	\]
	and
	\[
	u = \frac{\int^{\infty}_0vf\,dv}{\int^{\infty}_0f\,dv} \leq \frac{\rho^{2/3} \lt(\int^{\infty}_0v^3 f\,dv \rt)^{1/3}} {\int^{\infty}_0f\,dv} = \frac{ \lt(\int^{\infty}_0 v^3 f\,dv \rt)^{1/3}}{\lt(\int^{\infty}_0f\,dv\rt)^{1/3}},
	\]
	so that
	\begin{align*}
		u\int^{\infty}_0 v^2 f\,dv\leq\int^{\infty}_0v^3f\,dv.
	\end{align*}
Therefore, we have
	\begin{align*}
		\int^{\infty}_0v^2Q_i(f)\,dv 
		=\int^{\infty}_0\rho v^2(u-v)f\,dv
		= \rho \lt(u\int^{\infty}_0 v^2 f\,dv-\int^{\infty}_0v^3f\,dv \rt)\leq0.
	\end{align*}
\end{proof}
The following lemma is standard. We, however, record it in a separate lemma  since we need unusual assumption that $f(x,v,t)=0$ when $v=0$, so that the boundary term in the integration by parts vanishes.
\begin{lemma}Suppose that $f$ is a solution of \eqref{main_sys} with sufficient integrability. Furthermore, we  assume that $f$ is compactly supported in $x$ and $v$ at $t>0$ with $f(x,0,t)=0$. Then we have
\[
\int^{\infty}_0\partial_vQ_r(f) \,dv=0,\qquad\int^{\infty}_0v\partial_vQ_r(f)\,dv=0,
\]
and
\[
\frac12\int^{\infty}_0v^2Q_r(f)\,dv=-\int^{\infty}_0(u-v)^2f\,dv.
\]
\end{lemma}
\begin{proof}
Since the proof is straightforward, we only consider the third relation. From the assumption, we get $f(x,0,t)=0$ and $\lim_{v\rightarrow +\infty}v^3f(x,v,t) =0$, so that   	
\begin{align*}
\frac12\int^{\infty}_0v^2\partial_v\lt((u-v)f\rt)dv&=\frac{v^2}{2}(u-v)f\,\bigg|^{v=\infty}_{v=0}-\int^{\infty}_0v(u-v)f\,dv\cr
&=\int^{\infty}_0 (u-v)^2f\,dv-u\int^{\infty}_0(u-v)f\,dv\cr
&=\int^{\infty}_0 (u-v)^2f\,dv.
\end{align*}
In the last line, we used the following equality:
\[
\int^{\infty}_0(u-v)f\,dv=0.
\]
\end{proof}
From the above two lemmas, we have a priori energy estimates which follow directly by integrating (\ref{main_sys}) with respect to $(1,v,v^2)\,dv$, respectively.
\begin{lemma}\label{lem_lp}Suppose that $f$ is a solution of \eqref{main_sys} with sufficient integrability. Furthermore, we  assume that $f$ is compactly supported in $x$ and $v$ at $t>0$ with $f(x,0,t)=0$.  Then we have
\[
\frac{d}{dt} \int_{\R \times \R_+} f\,dxdv = 0, \quad \frac{d}{dt} \int_{\R \times \R_+} vf\,dxdv + \int_{\R \times \R_+} \rho(u-v)^2f\,dxdv = 0,
\]
and
\[
\frac12\frac{d}{dt}\int_{\R \times \R_+} v^2f\,dxdv + \int_{\R \times \R_+}  (u - v)^2 f\,dxdv \leq 0.
\]
\end{lemma}

%%%%%%%%%%%%%%%%%%%%%%%%%%%%%%%%%%%%%%%%%%%%%%%%%%%%%%%%%%%%%%%%%%%%%%%%%%%%%%%%%%%%%%
%
%  \subsection{Regularization and decoupling}
%
%%%%%%%%%%%%%%%%%%%%%%%%%%%%%%%%%%%%%%%%%%%%%%%%%%%%%%%%%%%%%%%%%%%%%%%%%%%%%%%%%%%%%%

\section{Global-in-time existence for a regularized equation}\label{sec_gl_reg}
In this section, we consider a regularized equation of \eqref{main_sys}. Inspired by \cite{KMT13}, we regularize the local vehicle velocity $u$ by using a mollifier $\theta_\e(x) = \e^{-1}\theta(x/\e)$ with $0 \leq \theta \in \mathcal{C}_0^\infty(\R)$ satisfying
\[
supp \, \theta \subseteq (-1,1) \quad \mbox{and} \quad \int_{\R} \theta(x)\,dx = 1.
\]
This removes the singularity in the relaxation term. More precisely, our regularized equation of \eqref{main_sys} is defined as follows:
\begin{align}\label{sys_reg}
\begin{split}
&\qquad\pa_t f_\e + v \pa_x f_\e + \pa_v((u^\e_\e-v)f_\e) = \frac{\rho_\e(u_\e - v)f_\e}{1+ \e \rho_\e (1 + u_\e)},
\end{split}
\end{align}
subject to regularized initial data
\[
f_\e(x,v,0) =: f_{0,\e}(x,v).
\]
Here, the regularized local vehicle velocity $u^\e_\e$ is defined by
\begin{align*}
u^\e_\e(x,t) &:= \lt(\int_{\R \times \R_+} \theta_\e(x-y) wf_\e(y,w,t)\,dydw\rt) \bigg / \lt(\e + \lt(\int_{\R \times \R_+} \theta_\e(x-y)f_\e(y,w,t)\,dydw \rt)\rt)\cr
&= \frac{((\rho_\e u_\e) \star \theta_\e)(x,t)}{\e + (\rho_\e\star \theta_\e)(x,t)},
\end{align*}
and
$f_{0,\varepsilon}$ denotes a smooth approxmiation of $f_0$ such that 
\begin{itemize}
\item[(i)] $f_{0,\varepsilon}$ is compactly 
supported and it is contained in $[C_{x,0},C_{x,1}]\times [C_{v,0}, C_{v,1}]$ with $C_{x,0} < C_{x,1}$ and $0< C_{v,0} < C_{v,1}$.
\item[(ii)] $f_{0,\varepsilon}$ converges strongly to $f_{0}$ in $L^{\infty}(\R \times \R_+)$ as $\e \to 0$.
\end{itemize}
Note that our main equation \eqref{main_sys} can be formally recovered from \eqref{sys_reg} in the limit $\e \to 0$.

In the following two subsections, we prove the proposition below on the global-in-time existence of weak solutions and some uniform bound estimates of the regularized equation \eqref{sys_reg}.
%%%%%%%%%%%%%%%%%%%%%%%%%%%%%%%%%%%%%%%%%%%%%%%%%%%%%%%%%%%%%%%%%%%%%%%%%%%%%%%%%%%%%%%%%%%%%%%%%%%%%%%%%%%%%%%%%%%%%%%%%
%
%
%                                Proposition
%
%
%%%%%%%%%%%%%%%%%%%%%%%%%%%%%%%%%%%%%%%%%%%%%%%%%%%%%%%%%%%%%%%%%%%%%%%%%%%%%%%%%%%%%%%%%%%%%%%%%%%%%%%%%%%%%%%%%%%%%%%%%
\begin{proposition}\label{prop_ext_reg} Let $T>0$. For any $\e >0$, there exists at least one weak solution $f_\e$ of the regularized equation \eqref{sys_reg} on the interval $[0,T]$ in the sense of Definition \ref{def_weak}. 
\end{proposition}

%%%%%%%%%%%%%%%%%%%%%%%%%%%%%%%%%%%%%%%%%%%%%%%%%%%%%%%%%%%%%%%%%%%%%%%%%%%%%%%%%%%%%%%%%%%%%%%%%%%%%%%%%%%%%%%%%%%%%%%%%
%
%
%
%
%%%%%%%%%%%%%%%%%%%%%%%%%%%%%%%%%%%%%%%%%%%%%%%%%%%%%%%%%%%%%%%%%%%%%%%%%%%%%%%%%%%%%%%%%%%%%%%%%%%%%%%%%%%%%%%%%%%%%%%%%%
\subsection{Approximated solutions}\label{sec_pro1} In order to obtain the existence of solutions to the regularized equation \eqref{sys_reg}, we first construct the approximated solutions in the following way:
\bq\label{sys_lin}
\pa_t f^{n+1}_\e + v \pa_x f^{n+1}_\e + \pa_v \lt((u^{\e,n}_\e - v)f^{n+1}_\e \rt) = \frac{\rho^n_\e(u^n_\e - v)f^{n+1}_\e}{1+ \e \rho^n_\e (1 + u^n_\e)},
\eq
with the initial data and first iteration step:
\[
f^n_\e(x,v,t)|_{t=0} = f_{0,\e}(x,v) \quad (x,v) \in \R \times \R_+ \quad \mbox{for all } n \geq 1
\]
and
\[
f^0_\e(x,v,t) = f_{0,\e}(x,v), \quad (x,v,t) \in \R \times \R_+ \times (0,T).
\]
Here 
\[
u^{\e,n}_\e(x,t) = \frac{((\rho^n_\e u^n_\e) \star \theta_\e)(x,t)}{\e + (\rho^n_\e \star \theta_\e)(x,t)} 
\]
with
\[
\rho^n_\e u^n_\e = \int_{\R_+} v f^n_\e\,dv \quad \mbox{and} \quad   \rho^n_\e = \int_{\R_+} f^n_\e\,dv.
\]
To get the existence of solutions to the approximated equation \eqref{sys_lin}, we need to estimate the support of $f^n_\e$, and for this, we introduce the following forward characteristics:
\[Z^{n+1}_\e(s) := (X^{n+1}_\e(s), V^{n+1}_\e(s)) := (X^{n+1}_\e(s;0,x,v), V^{n+1}_\e(s;0,x,v))\]
defined by
\begin{align}\label{eqn_rtra1}
\begin{aligned}
\frac{d}{ds}X^{n+1}_{\e}(s) &=V^{n+1}_\e(s), \quad 0 \leq s \leq T, \\[1mm]
\frac{d}{ds}V^{n+1}_\e(s) &=  u^{\e,n}_\e(X^{n+1}_\e(s),s)-V^{n+1}_\e(s),
\end{aligned}
\end{align}
with the initial data
\[
Z^{n+1}_\e(0)=(x,v) \in \R \times \R_+.
\]
Due to the regularization, the characteristics \eqref{eqn_rtra1} is well-defined, thus global-in-time existence of solutions to \eqref{sys_lin} can be obtained by a standard existence theory. It follows from \eqref{eqn_rtra1} that $V^{n+1}_\e(s)$ satisfies
\bq\label{eq_vn}
V^{n+1}_\e(t) = v e^{-t} + e^{-t}\int_0^t u^{\e,n}_\e(X^{n+1}_\e(s),s) e^s\,ds.
\eq
We now define 
\[
R_X[f]:= \max_{x \in \overline{supp_x f}}x \quad \mbox{and} \quad R_V[f]:= \max_{v \in \overline{supp_v f}}v,
\]
and
\[
r_X[f]:= \min_{x \in \overline{supp_x f}}x \quad \mbox{and} \quad r_V[f]:= \min_{v \in \overline{supp_v f}}v,
\]
where $supp_x f$ and $supp_v f$ represent $x$- and $v$-projections of $supp f$, respectively.

\begin{proposition}\label{prop_lin}For any $T>0$ and $n \in \N$, there exists a unique solution $f^n_\e$  of the regularized and linearized equation \eqref{sys_lin} such that
\begin{itemize}
\item[(i)] $f^n_\e$ is compactly supported in both $x$ and $v$ satisfying 
\[
C_{x,0} + C_{v,0}(1 - e^{-t}) \leq r_X[f^n_\e] \leq R_X[f^n_\e] \leq C_{x,1}+TC_{v,1}
\]
and
\[
e^{-t}C_{v,0}<r_V[f^n_\e] \leq R_V[f^n_\e] \leq C_{v,1}.
\]
Here the positive constants $C_{x,i}, C_{v,i}, i=1,2$ are appeared in the beginning of this section.
\item[(ii)] $f^n_\e \in L^\infty(0,T; L^\infty(\R \times \R_+))$ satisfies
\[
\sup_{0 \leq t \leq T}\sup_{n \in \N}\|f^n_\e(\cdot,\cdot,t)\|_{W^{1,\infty}} \leq C_{\e, T, f_{0,\e}},
\]
for $\e \in (0,1)$, where $C_{\e, T, f_{0,\e}}>0$ depends on $\e, T$, and $\|f_{0,\e}\|_{W^{1,\infty}}$, but independent of $n$.
\end{itemize}
\end{proposition}
\begin{remark}
Note that the size of the support does not depend on $\e$ or $n$. We also note that
(i) clearly implies $f^n_\e(x,0,t)=0$ for $[0,T]$. 
\end{remark}
\begin{proof}[Proof of Proposition \ref{prop_lin}]
%For the proof, we use the strong induction argument. Suppose that 
%\bq\label{ind_as}
%\sup_{0 \leq t \leq T} \max_{1 \leq m \leq n}\|f^m_\e(\cdot,\cdot,t)\|_{W^{1,\infty}} < M.
%\eq
(i) In view of the observation that the solution to (\ref{sys_lin}) is written in the characteristic formulation:
$$
f^{n+1}_{\e}(X^{n+1}_\e(t), V^{n+1}_\e(t))=e^{\int^t_0 (A^n_\e(s) + 1)\,ds}f_{0,\e}(x,v)
$$
where $A^n_\e(s)$ denotes
$$
A^n_\e(s):=\frac{\rho^n_\e(X^{n+1}_\e(s),s)\lt(u^n_\e(X^{n+1}_\e(s),s) - V^{n+1}_\e(s)\rt)}{1+ \e \rho^n_\e(X^{n+1}_\e(s),s) \lt(1 + u^n_\e (X^{n+1}_\e (s),s)\rt)}.
$$
We see that it is enough to derive the compactness of the characteristics.
We first estimate $R_V[f^{n+1}_\e(t)]$. Note that
\[
u^{\e,n}_\e(X^{n+1}_\e(s),s) \leq \frac{\int_{\R \times \R_+} vf^n_\e(X^{n+1}_\e(s)-y,v,t)\theta_\e(y)\,dydv}{\int_{\R \times \R_+} f^n_\e(X^{n+1}_\e(s)-y,v,t)\theta_\e(y)\,dydv} \leq R_V[f^n_\e(t)]. 
\]
Then it follows from \eqref{eq_vn} that 
\begin{align*}
V^{n+1}_\e(t) &= v e^{-t} + e^{-t}\int_0^t u^{\e,n}_\e(X^{n+1}_\e(s),s) e^s\,ds\cr
&\leq R_V[f_{0,\e}] e^{-t} + e^{-t} \int_0^t  R_V[f^n_\e(s)] e^s\,ds.
\end{align*}
This implies
\[
e^tR_V[f^{n+1}_\e(t)]  \leq R_V[f_{0,\e}] + \int_0^t  e^sR_V[f^n_\e(s)] \,ds.
\]
For notational simplicity, we set $A^n(t) :=e^t R_V[f^{n}_\e(t)]$ for $n \in \N$ and iterate this relation to get
\begin{align*}
A^{n}(t)&\leq R_V[f_{0,\e}]+\int^t_0A^{n-1}(s)\,ds\cr
&\leq R_V[f_{0,\e}]+\int^t_0R_V[f_{0,\e}]\,ds+\int^t_0\int^{s}_0{A^{n-2}(s_1)}\,ds_1ds\cr
&\leq R_V[f_{0,\e}] + t R_V[f_{0,\e}]+\int^t_0\int^{s}_0R_V[f_{0,\e}]\,ds_1ds+
\int^t_0\int^{s}_0\int^{s_1}_0{A^{n-3}(s_2)}\,ds_2ds_1ds\cr
&\leq \quad \cdots \cr
&\leq \left(\sum_{n=0}^{\infty}\frac{t^n}{n!}\right)R_V[f_{0,\e}]\cr
&\leq e^tR_V[f_{0,\e}].
\end{align*}
Thus we obtain
\bq\label{est_rv}
e^tR_V[f^n_\e(t)] \leq e^tR_V[f_{0,\e}], \quad \mbox{i.e.,} \quad R_V[f^n_\e(t)] \leq R_V[f_{0,\e}]<C_{v,1}
\eq
for all $n \in \N$.
%Thus, there exists $C>0$ independent of $\e$ and $n$ such that
%\[
%\sup_{0 \leq t \leq T} \sup_{n \in \N}R_V[f^n_\e(t)] \leq R_V[f_{0,\e}] \leq C_{v,1}.
%\]
Moreover, we get
\begin{align*}
	V^{n+1}_\e(t) &\geq v e^{-t}\geq e^{-t} r_V[f_{0,\e}],
\end{align*}
and subsequently, this asserts
\begin{align*}
r_V[f^n_\e(t)]> e^{-t} r_V[f_{0,\e}]\geq e^{-t}C_{v,0},
\end{align*}
due to the compact support assumption on $f_{0,\e}$. This gives the compactness of the velocity support and  $f^n_\e(x,0,t)=0$. The compactness of position directly
follows from this.\newline

\noindent(ii) We now estimate $\|f^{n+1}_\e\|_{L^\infty}$ uniformly in $n$. For $1 \leq p < \infty$, we find
\begin{align*}
&\frac{d}{dt}\|f^{n+1}_\e\|_{L^p}^p \cr
&\quad = p\int_{\R \times \R_+} (f^{n+1}_\e)^{p-1} \lt(-v\pa_x f^{n+1}_\e - \pa_v \lt((u^{\e,n}_\e - v)f^{n+1}_\e \rt) + \frac{\rho^n_\e(u^n_\e - v)f^{n+1}_\e}{1+ \e \rho^n_\e (1 + u^n_\e)}  \rt)dxdv\cr
&\quad =(p-1)\|f^{n+1}_\e\|_{L^p}^p + p \int_{\R \times \R_+} (f^{n+1}_\e)^p \frac{\rho^n_\e(u^n_\e - v)}{1+ \e \rho^n_\e (1 + u^n_\e)}\,dxdv.
\end{align*}
For the last term on the right side of the above equality, we use \eqref{est_rv} to obtain
\begin{align*}
\lt|\int_{\R \times \R_+} (f^{n+1}_\e)^p \frac{\rho^n_\e(u^n_\e - v)}{1+ \e \rho^n_\e (1 + u^n_\e)}\,dxdv\rt| &\leq \int_{\R \times \R_+} (f^{n+1}_\e)^p \frac{\rho^n_\e(u^n_\e + R_V[f^{n+1}_\e])}{1+ \e \rho^n_\e (1 + u^n_\e)}\,dxdv\cr
&\leq C\int_{\R \times \R_+} (f^{n+1}_\e)^p \frac{\rho^n_\e(u^n_\e + 1)}{1+ \e \rho^n_\e (1 + u^n_\e)}\,dxdv\cr
&\leq \frac C\e\|f^{n+1}_\e\|_{L^p}^p,
\end{align*}
where $C>0$ is independent of $n, p$ and $\e$. Thus we get
\[
\frac{d}{dt}\|f^{n+1}_\e\|_{L^p} \leq C\lt(1 + \frac1\e\rt)\|f^{n+1}_\e\|_{L^p},
\]
where $C>0$ is independent of $n, p$ and $\e$. We now use Gr\"onwall's lemma and send $p\to\infty$ to have
\[
\|f^{n+1}_\e\|_{L^\infty} \leq \|f_{0,\e}\|_{L^\infty} e^{C(1 + \e^{-1})T},
\]
where $C>0$ is independent of $n$ and $\e$.
%\[
%\rho^n_\e u^n_\e = \int_{\R_+} v f^n_\e\,dv \leq R_V[f^n_\e(s)]^2 \|f^n_\e\|_{L^\infty} \leq C\|f^n_\e\|_{L^\infty} 
%\]
%and
%\[
%\rho^n_\e =  \int_{\R_+} f^n_\e\,dv \leq R_V[f^n_\e(s)] \|f^n_\e\|_{L^\infty} \leq C\|f^n_\e\|_{L^\infty} 
%\]

We next differentiate \eqref{sys_lin} with respect to $x$ to find
\[
\pa_t \pa_x f^{n+1}_\e + v \pa_{xx} f^{n+1}_\e + \pa_{xv} \lt((u^{\e,n}_\e - v)f^{n+1}_\e \rt) = \pa_x\lt(\frac{\rho^n_\e(u^n_\e - v)}{1+ \e \rho^n_\e (1 + u^n_\e)}f^{n+1}_\e\rt).
\]
Then we estimate
\begin{align*}
\frac{d}{dt}\|\pa_x f^{n+1}_\e\|_{L^p}^p &= p\int_{\R \times \R_+} |\pa_x f^{n+1}_\e|^{p-2} \pa_x f^{n+1}_\e\lt( -v\pa_{xx} f^{n+1}_\e - \pa_{xv} \lt((u^{\e,n}_\e - v)f^{n+1}_\e\rt) \rt)dxdv\cr
&\quad +p\int_{\R \times \R_+} |\pa_x f^{n+1}_\e|^{p-2} \pa_x f^{n+1}_\e \pa_x\lt(\frac{\rho^n_\e(u^n_\e - v)}{1+ \e \rho^n_\e (1 + u^n_\e)}f^{n+1}_\e\rt) dxdv\cr
&=: I_1 + I_2 + I_3.
\end{align*}
For $I_1$, we use the integration by parts to obtain
\[
I_1 = -\int_{\R \times \R_+} \pa_x (|\pa_x f^{n+1}_\e|^p) v\,dxdv = 0.
\]
For $I_2$, we get
\begin{align*}
I_2 &= - p\int_{\R \times \R_+} |\pa_x f^{n+1}_\e|^{p-2} \pa_x f^{n+1}_\e \lt(-\pa_x f^{n+1}_\e + \pa_x u^{\e,n}_\e \pa_v f^{n+1}_\e + (u^{\e,n}_\e - v) \pa_{xv} f^{n+1}_\e   \rt) dxdv\cr
&\leq p \|\pa_x f^{n+1}_\e\|_{L^p}^p + C_\e p \|f^n_\e\|_{L^\infty}\|\pa_x f^{n+1}_\e\|_{L^p}^{p-1}\|\pa_v f^{n+1}_\e\|_{L^p} + \|\pa_x f^{n+1}_\e\|_{L^p}^p.
\end{align*}
Here we used
\begin{align*}
|\pa_x u^{\e,n}_\e| &= \lt|\frac{1}{(\e + \rho^n_\e \star \theta_\e)^2}\lt(\pa_x((\rho^n_\e u^n_\e) \star \theta_\e)(\rho^n_\e \star \theta_\e) - ((\rho^n_\e u^n_\e) \star \theta_\e)(\pa_x\rho^n_\e \star \theta_\e) \rt) \rt|\cr
&\leq \lt|\frac{(\rho^n_\e u^n_\e) \star \pa_x \theta_\e}{\e + \rho^n_\e \star \theta_\e} \rt| + C\lt|\frac{\rho^n_\e \star \pa_x \theta_\e}{\e + \rho^n_\e \star \theta_\e} \rt|\cr
&\leq C_\e \|f^n_\e\|_{L^\infty}
\end{align*}
and
\begin{align*}
&- p\int_{\R \times \R_+} |\pa_x f^{n+1}_\e|^{p-2} \pa_x f^{n+1}_\e (u^{\e,n}_\e - v) \pa_{xv} f^{n+1}_\e dxdv \cr
&\quad = \int_{\R \times \R_+} \pa_v( |\pa_x f^{n+1}_\e|^p) (u^{\e,n}_\e - v) \,dxdv\cr
&\quad = \|\pa_x f^{n+1}_\e\|_{L^p}^p.
\end{align*}
For the estimate of $I_3$, we notice from \eqref{est_rv} that
\[
|\pa_x (\rho^n_\e u^n_\e)| = \lt|\int_{\R_+} v \pa_x f^n_\e\,dv\rt| \leq C\|\pa_x f^n_\e\|_{L^p}
\]
and
\[
|\pa_x  \rho^n_\e| = \lt|\int_{\R_+} \pa_x f^n_\e\,dv\rt| \leq C\|\pa_x f^n_\e\|_{L^p}.
\]
This gives
\begin{align*}
\lt|\pa_x \lt(\frac{\rho^n_\e(u^n_\e - v)}{1+ \e \rho^n_\e (1 + u^n_\e)} \rt)\rt| &= \lt|\frac{\pa_x (\rho^n_\e u^n_\e) - v\pa_x  \rho^n_\e}{(1+ \e \rho^n_\e (1 + u^n_\e))} - \frac{\e\rho^n_\e(u^n_\e - v)((\pa_x  \rho^n_\e + \pa_x (\rho^n_\e u^n_\e)))}{(1+ \e \rho^n_\e (1 + u^n_\e))^2}\rt|\cr
&\leq C(1 + |v|)\|\pa_x f^n_\e\|_{L^p},
\end{align*}
for $\e \in (0,1)$, and subsequently this asserts
\begin{align*}
I_3 &= p\int_{\R \times \R_+} |\pa_x f^{n+1}_\e|^{p-2} \pa_x f^{n+1}_\e  \pa_x \lt(\frac{\rho^n_\e(u^n_\e - v)}{1+ \e \rho^n_\e (1 + u^n_\e)} \rt)f^{n+1}_\e\,dxdv  \cr
&\quad+ p\int_{\R \times \R_+} |\pa_x f^{n+1}_\e|^{p-2} \pa_x f^{n+1}_\e  \frac{\rho^n_\e(u^n_\e - v)}{1+ \e \rho^n_\e (1 + u^n_\e)}\pa_x f^{n+1}_\e \,dxdv\cr
&\leq Cp\|\pa_x f^{n+1}_\e\|_{L^p}^{p-1}\|f^{n+1}_\e\|_{L^p} + C_\e p\|\pa_x f^{n+1}_\e\|_{L^p}^p.
\end{align*}
Thus we obtain
\[
\frac{d}{dt}\|\pa_x f^{n+1}_\e\|_{L^p} \leq C_{\e, T, f_{0,\e}}\lt(1 + \frac1p\rt)\|f^{n+1}_\e\|_{W^{1,p}}, 
\]
where $C_{\e, T, f_{0,\e}} > 0$ depends on $\e, T, $ and $\|f_{0,\e}\|_{L^\infty}$. Similarly, we also find 
\begin{align*}
\frac{d}{dt}\|\pa_v f^{n+1}_\e\|_{L^p}^p &= p\int_{\R \times \R_+} |\pa_v f^{n+1}_\e|^{p-2} \pa_v f^{n+1}_\e \lt( - \pa_x f^{n+1}_\e - \pa_{vv} ( (u^{\e,n}_\e - v)f^{n+1}_\e)\rt) dxdv\cr
&\quad + p\int_{\R \times \R_+} |\pa_v f^{n+1}_\e|^{p-2} \pa_v f^{n+1}_\e \frac{\rho^n_\e}{1+ \e \rho^n_\e (1 + u^n_\e)} \pa_v((u^n_\e - v)f^{n+1}_\e)\, dxdv\cr
&=: J_1 + J_2 + J_3.
\end{align*}
Here $J_i,i=1,2,3$ can be estimated as follows.
\begin{align*}
J_1 &\leq p\|\pa_v f^{n+1}_\e\|_{L^p}^{p-1}\|\pa_x f^{n+1}_\e\|_{L^p},\cr
J_2&= -p\int_{\R \times \R_+} |\pa_v f^{n+1}_\e|^{p-2} \pa_v f^{n+1}_\e \pa_v (-f^{n+1}_\e + (u^{\e,n}_\e - v)\pa_v f^{n+1}_\e )\,dxdv\cr
&=2p\|\pa_v f^{n+1}_\e\|_{L^p}^p - \|\pa_v f^{n+1}_\e\|_{L^p}^p,\cr
J_3&=p\int_{\R \times \R_+} |\pa_v f^{n+1}_\e|^{p-2} \pa_v f^{n+1}_\e \frac{\rho^n_\e}{1+ \e \rho^n_\e (1 + u^n_\e)} ( -f^{n+1}_\e + (u^n_\e - v)\pa_v f^{n+1}_\e)\, dxdv\cr
&\leq C_\e p\|\pa_v f^{n+1}_\e\|_{L^p}^{p-1}\|f^{n+1}_\e\|_{L^p} + C_\e p\|\pa_v f^{n+1}_\e\|_{L^p}^p.
\end{align*}
Hence we have
\[
\frac{d}{dt}\|\pa_v f^{n+1}_\e\|_{L^p} \leq C_\e \|f^{n+1}_\e\|_{W^{1,p}}.
\]
We now combine all of the above estimates to arrive at
\[
\frac{d}{dt}\|f^{n+1}_\e\|_{W^{1,p}} \leq C_{\e, T, f_{0,\e}}\lt(1+\frac1p\rt)\|f^{n+1}_\e\|_{W^{1,p}}. 
\]
Then, applying Gr\"onwall's lemma and letting $p \to \infty$, we get
\[
\|f^{n+1}_\e(\cdot,\cdot,t)\|_{W^{1,\infty}} \leq \|f_{0,\e}\|_{W^{1,\infty}} e^{C_{\e, T, f_{0,\e}}},
\]
for $t \in [0,T]$, where $C_{\e, T, f_{0,\e}} > 0$ is independent of $n$. 
\end{proof}
%%%%%%%%%%%%%%%%%%%%%%%%%%%%%%%%%%%%%%%%%%%%%%%%%%%%%%%%%%%%%%%%%%%%%%%%%%
%
%  \subsection{Proof of Proposition \ref{prop_ext_reg}
%
%%%%%%%%%%%%%%%%%%%%%%%%%%%%%%%%%%%%%%%%%%%%%%%%%%%%%%%%%%%%%%%%%%%%%%%%%%%%%%%%%%%%%%
\subsection{Proof of Proposition \ref{prop_ext_reg}: Existence of weak solutions of \eqref{sys_reg}} 
In this subsection, we establish that the global-in-time existence of weak solutions to the regularized equation \eqref{sys_reg}. For this, we first show that the approximation sequence $\{f^n\}_{n \in \N}$ is Cauchy. It follows from \eqref{sys_lin} that
\begin{align*}
&\pa_t (f^{n+1}_\e-f^{n}_\e) + v\pa_x (f^{n+1}_\e-f^{n}_\e) + \pa_v ((u^{\e,n}_\e - v)(f^{n+1}_\e-f^{n}_\e)) + (u^{\e,n}_\e - u^{\e,n-1}_\e)\pa_v f^n_\e\cr
&\quad = \frac{\rho^n_\e (u^n_\e - v)}{1+ \e \rho^n_\e (1 + u^n_\e)}(f^{n+1}_\e-f^{n}_\e) + \frac{f^n_\e}{1+ \e \rho^n_\e (1 + u^n_\e)} \lt(\rho^n_\e u^n_\e - \rho^{n-1}_\e u^{n-1}_\e - v(\rho^n_\e - \rho^{n-1}_\e)\rt)\cr
&\qquad + f^n_\e \rho^{n-1}_\e(u^{n-1}_\e - v)\lt(\frac{1}{1+ \e \rho^n_\e (1 + u^n_\e)} - \frac{1}{1+ \e \rho^{n-1}_\e (1 + u^{n-1}_\e)} \rt).
\end{align*}
Then we obtain
\begin{align*}
&\frac{d}{dt}\|f^{n+1}_\e-f^{n}_\e\|_{L^p}^p \cr
&\quad = -p \int_{\R \times \R_+} |f^{n+1}_\e-f^{n}_\e|^{p-2} (f^{n+1}_\e-f^{n}_\e) \pa_v ((u^{\e,n}_\e - v)(f^{n+1}_\e-f^{n}_\e))\, dxdv\cr
&\qquad -p \int_{\R \times \R_+} |f^{n+1}_\e-f^{n}_\e|^{p-2} (f^{n+1}_\e-f^{n}_\e) (u^{\e,n}_\e - u^{\e,n-1}_\e)\pa_v f^n_\e dxdv\cr
&\qquad+ p \int_{\R \times \R_+} |f^{n+1}_\e-f^{n}_\e|^{p-2} (f^{n+1}_\e-f^{n}_\e)\rho^n_\e (u^n_\e - v)(f^{n+1}_\e-f^{n}_\e) \,dxdv \cr
&\qquad+ p \int_{\R \times \R_+} |f^{n+1}_\e-f^{n}_\e|^{p-2} (f^{n+1}_\e-f^{n}_\e) f^n_\e (\rho^n_\e u^n_\e - \rho^{n-1}_\e u^{n-1}_\e)\,dxdv \cr
&\qquad -p \int_{\R \times \R_+} |f^{n+1}_\e-f^{n}_\e|^{p-2} (f^{n+1}_\e-f^{n}_\e) v(\rho^n_\e - \rho^{n-1}_\e))\, dxdv\cr
&\qquad +p \int_{\R \times \R_+} |f^{n+1}_\e-f^{n}_\e|^{p-2} (f^{n+1}_\e-f^{n}_\e)\cr
&\hspace{3cm} \times \frac{\e f^n_\e \rho^{n-1}_\e(u^{n-1}_\e - v)(\rho^{n-1}_\e - \rho^n_\e + \rho^{n-1}_\e u^{n-1}_\e - \rho^n_\e u^n_\e)}{(1+ \e \rho^n_\e (1 + u^n_\e))(1+ \e \rho^{n-1}_\e (1 + u^{n-1}_\e))}\, dxdv\cr
&\quad =:\sum_{i=1}^6 K_i,
\end{align*}
where $K_i, i=1,\dots,6$ can be estimated as follows.
\begin{align*}
K_1 &= (p+1)\|f^{n+1}_\e-f^{n}_\e\|_{L^p}^p,\cr
K_2 &\leq C_{\e, T, f_{0,\e}}p\|f^{n+1}_\e-f^{n}_\e\|_{L^p}^{p-1}\|f^n_\e-f^{n-1}_\e\|_{L^\infty},\cr
K_3 &\leq C_{\e, T, f_{0,\e}}p\|f^{n+1}_\e-f^{n}_\e\|_{L^p}^p,\cr
K_4 &\leq C_{\e, T, f_{0,\e}}p\|f^{n+1}_\e-f^{n}_\e\|_{L^p}^{p-1}\|f^n_\e-f^{n-1}_\e\|_{L^\infty},\cr
K_5 &\leq Cp\|f^{n+1}_\e-f^{n}_\e\|_{L^p}^{p-1}\|f^n_\e-f^{n-1}_\e\|_{L^\infty},\cr
K_6 &\leq C_{\e, T, f_{0,\e}}p\|f^{n+1}_\e-f^{n}_\e\|_{L^p}^{p-1}\|f^n_\e-f^{n-1}_\e\|_{L^\infty}.
\end{align*}
Here we used
\begin{align}\label{n1}
|\rho^n_\e - \rho^{n-1}_\e| &= \lt|\int_{\R_+} (f^n_\e - f^{n-1}_\e) \,dv\rt| \leq C\|f^n_\e - f^{n-1}_\e\|_{L^\infty},\cr
|\rho^n_\e u^n_\e - \rho^{n-1}_\e u^{n-1}_\e| &= \lt|\int_{\R_+} v(f^n_\e - f^{n-1}_\e) \,dv \rt| \leq C\|f^n_\e - f^{n-1}_\e\|_{L^\infty},
\end{align}
and
\begin{align}\label{n2}
&|u^{\e,n}_\e - u^{\e,n-1}_\e| \cr
&\quad = \lt|\frac{(\rho^n_\e u^n_\e - \rho^{n-1}_\e u^{n-1}_\e) \star \theta_\e}{\e + \rho^n_\e}  + \frac{((\rho^{n-1}_\e u^{n-1}_\e)\star\theta_\e) (\rho^n - \rho^{n-1})\star\theta_\e}{(\e + \rho^n_\e)(\e + \rho^{n-1}_\e)}\rt|\cr
&\quad \leq C_{\e, T, f_{0,\e}}\|f^n_\e - f^{n-1}_\e\|_{L^\infty}.
\end{align}
Thus we have
\[
\frac{d}{dt}\|f^{n+1}_\e-f^{n}_\e\|_{L^p} \leq C_{\e, T, f_{0,\e}}\lt(\|f^{n+1}_\e-f^{n}_\e\|_{L^p} + \|f^n_\e-f^{n-1}_\e\|_{L^\infty}\rt),
\]
where $C_{\e, T, f_{0,\e}}>0$ is independent of $n$. This, together with applying Gr\"onwall's lemma and passing $p \to \infty$, yields
\[
\|(f^{n+1}_\e-f^{n}_\e)(\cdot,\cdot,t)\|_{L^\infty} \leq C_{\e, T, f_{0,\e}}\int_0^t \|(f^n_\e-f^{n-1}_\e)(\cdot,\cdot,s)\|_{L^\infty}\,ds.
\]
This concludes that $f^n$ is Cauchy in $L^\infty(0,T; L^\infty(\R \times \R_+))$ from which, for a fixed $\e > 0$, there exists a limiting function $f_\e$ such that
\[
\sup_{0 \leq t \leq T} \|(f^{n}_\e-f_\e)(\cdot,\cdot,t)\|_{L^{\infty}} \rightarrow 0,
\]
as $n\to\infty$. Due to \eqref{n1} and \eqref{n2}, we can easily show that the liming function $f_\e$ solves the regularized equation \eqref{sys_reg}.

\subsection{Proof of Proposition \ref{prop_ext_reg}: Uniform-in-$\e$ bound estimates}

In this part, we establish several uniform-in-$\e$ estimates for $f_\e$. 

\subsubsection{Support estimates} 
We recall from Proposition \ref{prop_lin} that
\begin{align*}
	f^n_\e(x,v)=0 ~\mbox{ if }~ (x,v)\in  (\R \times \R_+) \setminus \lt([C_{x,0}, C_{x,1}+TC_{v,1}] \times [ e^{-T}C_{v,0}, C_{v,1}]\rt)
\end{align*}
for $0\leq t\leq T$. Since $C_{x,i}$ and $C_{v,i}$ $(i=0,1)$ do not depend on $\e$ and $f^n_\e$ converges uniformly to $f_\e$,
this implies that the support of $f_\e$ is also contained in the same area: 
\[
[C_{x,0}, C_{x,1}+TC_{v,1}] \times [ e^{-T}C_{v,0}, C_{v,1}].
\]
That is,
\[
R_V[f_\e(t)]\leq C_{v,1},\quad R_X[f_\e(t)]\leq C_{x,1} +TC_{v,1}
\]
and
\[
r_X[f_\e(t)]\geq C_{x,0}, \quad r_V[f_\e(t)]\geq e^{-T}C_{v,0}.
\]
Note that this automatically implies $f_\e(x,0,t)=0$.

\subsubsection{Uniform bounds of the moment and energy estimates} 

We first provide the uniform energy estimate. It follows from \cite[Lemma 2.5]{KMT13} that 
\bq\label{gd_ineq}
\sup_{y \in \R}\int_{\R} \theta_\e(x-y) \frac{\rho_\e(x)}{\theta_\e \star \rho_\e(x)}\,dx \leq C,
\eq
where $C>0$ is independent of $\e$. On the other hand, a straightforward computation yields 
\begin{align*}
\frac12\frac{d}{dt}\int_{\R \times \R_+} v^2 f_\e\,dxdv &= \int_{\R \times \R_+} v \cdot  (u^\e_\e - v)f_\e\,dxdv + \int_{\R \times \R_+} \frac{v^2 \rho_\e(u_\e - v)f_\e}{1+ \e \rho_\e (1 + u_\e)}\,dxdv \cr
&\leq \frac12\int_{\R} |u^\e_\e|^2 \rho_\e\,dx - \frac12\int_{\R \times \R_+} v^2f_\e\,dxdv\cr
&\leq \frac12  \int_{\R} |u^\e_\e|^2 \rho_\e\,dx,
\end{align*}
where we used Lemma \ref{lem_lp}:
\begin{align*}
\int_{\R \times \R_+} \frac{v^2 \rho_\e(u_\e - v)f_\e}{1+ \e \rho_\e (1 + u_\e)}\,dxdv&=\int_\R \frac{\rho_\e}{1+ \e \rho^n_\e (1 + u^n_\e)}\lt(u_\e \int_{\R_+} v^2 f_\e\,dv - \int_{\R_+} v^3 f_\e\,dv \rt)dx\cr
& \leq 0,
\end{align*}
and 
\[
v\cdot(u^{\varepsilon}_{\varepsilon}-v)\leq \frac{(u^{\varepsilon}_{\varepsilon})^2+v^2}{2}-v^2=
\frac{(u^{\varepsilon}_{\varepsilon})^2-v^2}{2}.
\]
Note that 
$$\begin{aligned}
|u^\e_\e(x,t)|^2 &\leq \lt|\frac{\int_{\R \times \R_+} \theta_\e(x-y) wf_\e(y,w,t)\,dydw}{\theta_\e \star \rho_\e(x,t)} \rt|^2 \cr
&\leq \frac{\int_{\R \times \R_+} \theta_\e(x-y)w^2 f_\e(y,w,t)\,dydw}{\theta_\e \star \rho_\e(x,t)},
\end{aligned}$$
and this together with \eqref{gd_ineq} gives
\begin{align}\label{below}
\begin{aligned}
\int_{\R} \rho_\e |u^\e_\e|^2\,dx &\leq \int_{\R \times \R_+} \lt(\int_{\R_+} \theta_\e(x-y) \frac{\rho_\e(x)}{\theta_\e \star \rho_\e(x)}\,dx  \rt) w^2 f_\e(y,w)\,dydw\cr
&\leq C\int_{\R \times \R_+} v^2 f_\e(x,v)\,dxdv,
\end{aligned}
\end{align}
where $C>0$ is independent of $\e$. Hence we have
\[
\frac{d}{dt}\int_{\R \times \R_+} v^2 f_\e\,dxdv \leq C\int_{\R \times \R_+} v^2 f_\e\,dxdv,
\]
i.e.,
\[
\int_{\R \times \R_+} v^2 f_\e\,dxdv \leq C\int_{\R \times \R_+} v^2 f_{0,\e}\,dxdv,
\]
where $C>0$ is independent of $\e$.

We next show the moment estimate. By multiplying the regularized equation \eqref{sys_reg} by $v$ and integrating the resulting equation over $\R \times \R_+$, we find
\[
\frac{d}{dt} \int_{\R \times \R_+} vf_\e\,dxdv = -\int_{\R \times \R_+} v\partial_v((u^\e_\e-v)f_\e)\,dxdv + \int_{\R \times \R_+} \frac{v\rho_\e(u_\e - v)f_\e}{1+ \e \rho_\e (1 + u_\e)}\,dxdv.
\]
Here the last term on the right hand side of the above equation can be estimated as 
\[
\int_{\R \times \R_+} \frac{v\rho_\e(u_\e - v)f_\e}{1+ \e \rho_\e (1 + u_\e)}\,dxdv = -\int_{\R \times \R_+} \frac{\rho_\e(u_\e - v)^2 f_\e}{1+ \e \rho_\e (1 + u_\e)}\,dxdv,
\]
due to 
\[
\int_{\R_+} (u_\e - v)f_\e\,dv = 0.
\]
For the estimate of the first term on the right hand side, we use the uniform estimate above, \eqref{below}, and the fact that 
\[
\int_{\R} \rho_\e u_\e\,dx = \int_{\R \times \R_+} vf_\e\,dxdv \geq 0
\]
to get
\begin{align*}
-\int_{\R \times \R_+} v\partial_v ((u^\e_\e-v)f_\e)\,dxdv&=\int_{\R \times \R_+}(u^\e_\e-v)f_\e\,dxdv\cr
&=\int_{\R}(\rho_\e u^\e_\e-\rho_\e u_\e)\,dx\cr 
&\leq \int_{\R}\rho_\e u^\e_\e \,dx\cr
&\leq \frac{1}{2}\int_{\R}\rho_\e \,dx +\frac{1}{2}\int_{\R}\rho_\e|u^\e_\e|^2 \,dx\cr
&\leq C\int_{\R\times \R_+}(1+v^2)f_{\e}\,dxdv\cr
&\leq C\int_{\R\times \R_+}(1+v^2)f_{0,\e}\,dxdv,
\end{align*}
where $C>0$ is independent of $\e$. Thus, we obtain
\begin{align*}
\frac{d}{dt} \int_{\R \times \R_+} vf_\e\,dxdv  \leq C\int_{\R\times \R_+}(1+v^2)f_{0,\e}\,dxdv -\int_{\R \times \R_+} \frac{\rho_\e(u_\e - v)^2 f_\e}{1+ \e \rho_\e (1 + u_\e)}\,dxdv.
\end{align*}
By integrating the above differential inequality with respect to time, we have
\begin{align}\label{mom_est}
\begin{aligned}
&\int_{\R \times \R_+} vf_\e\,dxdv+\int^t_0\int_{\R \times \R_+} \frac{\rho_\e(u_\e - v)^2 f_\e}{1+ \e \rho_\e (1 + u_\e)}\,dxdvds\cr
&\quad \leq \int_{\R\times \R_+}vf_{0,\e}\,dxdv + C_T\int_{\R\times \R_+}(1+v^2)f_{0,\e}\,dxdv\cr
&\quad \leq C_T\int_{\R\times \R_+}(1+v^2)f_{0,\e}\,dxdv,
\end{aligned}
\end{align}
where $C>0$ is independent of $\e$.

\subsubsection{Uniform $L^\infty$-bound of $f_\e$} For $1 \leq p < \infty$, we find
\begin{align*}
\frac{d}{dt}\|f_\e\|_{L^p}^p &= p\int_{\R \times \R_+} (f_\e)^{p-1} \lt( - \pa_v \lt((u^{\e}_\e - v)f_\e \rt) + \frac{\rho_\e(u_\e - v)f_\e}{1+ \e \rho_\e (1 + u_\e)}  \rt)dxdv\cr
&=: L_1 + L_2.
\end{align*}
Here $L_1$ can be easily estimated as
\[
L_1 = (p-1)\|f_\e\|_{L^p}^p.
\]
For the estimate of $L_2$, we obtain
\begin{align*}
L_2 &= p\int_{\R \times \R_+} (f_\e)^p\frac{\rho_\e(u_\e - v)}{1+ \e \rho_\e (1 + u_\e)}\,dxdv \cr
&\leq p\|f_\e\|_{L^\infty}^{p - 1/2}\int_{\R \times \R_+} (f_\e)^{1/2}\frac{\rho_\e|u_\e - v|}{1+ \e \rho_\e (1 + u_\e)}\,dxdv\cr
&\leq Cp\|f_\e\|_{L^\infty}^{p - 1/2}\lt(\int_\R \frac{\rho_\e}{1+ \e \rho_\e (1 + u_\e)}\,dx \rt)^{1/2}\lt(\int_{\R \times \R_+} \frac{\rho_\e(u_\e - v)^2 f_\e}{1+ \e \rho_\e (1 + u_\e)}\,dxdv \rt)^{1/2}\cr
&\leq Cp\|f_\e\|_{L^\infty}^{p - 1/2}\lt(\int_{\R \times \R_+} \frac{\rho_\e(u_\e - v)^2 f_\e}{1+ \e \rho_\e (1 + u_\e)}\,dxdv \rt)^{1/2},
\end{align*}
where we used the uniform bound estimates of supports of $f_\e$. Thus we have
\begin{align*}
\frac{d}{dt}\|f_\e\|_{L^p} &\leq C\|f_\e\|_{L^p} + C\|f_\e\|_{L^p}^{1/2}\lt(\int_{\R \times \R_+} \frac{\rho_\e(u_\e - v)^2 f_\e}{1+ \e \rho_\e (1 + u_\e)}\,dxdv \rt)^{1/2} \cr
&\leq C\|f_\e\|_{L^p} + \int_{\R \times \R_+} \frac{\rho_\e(u_\e - v)^2 f_\e}{1+ \e \rho_\e (1 + u_\e)}\,dxdv.
\end{align*}
We now use Gr\"onwall's lemma and pass to the limit $p \to \infty$ to conclude
\begin{align*}
\|f_\e(\cdot,\cdot,t)\|_{L^\infty} &\leq C\|f_{0,\e}\|_{L^\infty} + C\int_0^t \int_{\R \times \R_+} \frac{\rho_\e(u_\e - v)^2 f_\e}{1+ \e \rho_\e (1 + u_\e)}\,dxdvds\cr
&\leq C\|f_{0,\e}\|_{L^\infty} + C_T\int_{\R \times \R_+} (1 + v^2)f_{0,\e}\,dxdv,
\end{align*}
where $C>0$ is independent of $\e$ and we used the uniform moment estimate \eqref{mom_est}.

%%%%%%%%%%%%%%%%%%%%%%%%%%%%%%%%%%%%%%%%%%%%%%%%%%%%%%%%%%%%%%%%%%%%%%%%%%%%%%%%%%%%%%%%%%%%%%%%%%%%%%%%%%%%%%%%%%%%%%%
%
%
%
%
%
%
%%%%%%%%%%%%%%%%%%%%%%%%%%%%%%%%%%%%%%%%%%%%%%%%%%%%%%%%%%%%%%%%%%%%%%%%%%%%%%%%%%%%%%%%%%%%%%%%%%%%%%%%%%%%%%%%%%%%%%%%%
\section{Global existence of weak solutions: Proof of Theorem \ref{thm_main}}\label{sec_gl_weak}
In this section, we provide the details of proof of Theorem \ref{thm_main}. 
\subsection{Strong compactness of $\rho_\e$ and $u_\e$}
From the argument in the previous section, we see that there exists $f\in L^1( \R\times \R_+ \times (0,{T}))$ such that $f_{\e}$, $f_{\e}v$   converge to $f$, $fv$ weakly in $L^1( \R\times \R_+ \times (0,{T}))$
respectively, which also implies
\[
\rho_\e = \int_{\R_+} f_\e \,dv \rightharpoonup \int_{\R_+} f \,dv = \rho \quad \mbox{and} \quad \rho_\e u_\e = \int_{\R_+} vf_\e \,dv \rightharpoonup \int_{\R_+} vf \,dv = \rho u,
\]
 in $L^1(\R \times (0,T))$, respectively.
Thanks to Lemma \ref{lem_veloa},  the above convergences actually are strong, which also give the almost everywhere
convergences of the macroscopic fields $\rho_\e$ and $\rho_\e u_\e$:
\bq\label{ae}
\rho_{\e} \rightarrow \rho \quad \mbox{a.e on } \R\times [0,T] \quad \mbox{and} \quad
\rho_{\e}u_{\e } \rightarrow \rho u \quad  \mbox{a.e on } \R\times [0,T].
\eq
%%%%%%%%%%%%%%%%%%%%%%%%%%%%%%%%%%%%%%%%%%%%%%%%%%%%%%%%%%%%%%%%%%%%%%%%%%%%%%%%%%%%%%%%%%%%%%%%%%%%%%%%%%%%%%%%%%%%%%%%%%%%%%%%%%%%%%%%%%%%%

\subsection{$f_\e u^\e_\e$ converges to $fu$ in $L^\infty(0,T; L^p(\R \times \R_+))$}
In this part, we show that 
\[
f_\e u^\e_\e \rightharpoonup f u \quad \mbox{in} \quad L^\infty(0,T; L^p(\R \times \R_+)) \quad \mbox{for} \quad p \in \lt(1, 3/2 \rt).
\]
Even though the proof is almost the same with \cite{KMT13}, we briefly present it for the completeness of our work. It follows from \eqref{ae} together with \cite[Lemma 6]{KMT13} that
\[
(\rho_\e u_\e)\star \theta_\e \to \rho u, \quad \rho_\e \star \theta_\e \to \rho \quad \mbox{a.e.} \quad \mbox{and} \quad L^p(\R \times (0,T))\mbox{-strong},
\]
up to a subsequence, for all $p \in (1, 3/2)$. Let
\[
\rho^\varphi_\e = \int_{\R_+} f_\e \varphi(v)\,dv,
\]
for a given test function $\varphi(v)$. Consider a test function $\psi(x,v,t) := \phi(x,t)\varphi(v)$ with $\phi \in \mc^\infty_c(\R \times (0,T))$ and $\varphi \in \mc^\infty_c(\R_+)$. Then we find
\[
\int_0^{T} \int_{\R \times \R_+} f_\e u^\e_\e \psi\,dxdvdt = \int_0^{T} \int_{\R} u^\e_\e \rho^\varphi_\e \phi \,dxdt.
\]
Note that if $p \in (1, 3/2)$, then $p/(2-p) \in (1, 3)$, and this gives the following uniform boundedness in $\e$:
\[
\|u^\e_\e \rho^\varphi_\e\|_{L^p} \leq \|\varphi\|_{L^\infty}\|\rho_\e\|^{1/2}_{L^{p/(2-p)}} \|\sqrt{\rho_\e} u^\e_\e\|_{L^2} < \infty.
\]
This implies that there exists a function $m \in L^\infty(0,T; L^p(\R))$ such that
\[
u^\e_\e \rho^\varphi_\e \rightharpoonup m \quad \mbox{in} \quad L^\infty(0,T; L^p(\R)) \quad \mbox{for all} \quad p \in \lt(1, 3/2 \rt)
\]
up to a subsequence. We now prove 
\[
m = u \rho^\varphi, \quad \mbox{where} \quad \rho^\varphi = \int_{\R_+} f \varphi\,dv \mbox{ and } \rho u = \int_{\R_+} v f\,dv.
\]
By using the set $E_R^0:= \lt\{ (x,t) \in (-R,R) \times (0, T)\,:\, \rho(x,t) =0 \rt\}$, we estimate
\[
\|u^\e_\e \rho^\varphi_\e\|_{L^p(E_R^0)} \leq C\|\rho_\e\|^{1/2}_{L^{p/(2-p)}(E_R^0)} \to 0 \quad \mbox{as} \quad \e \to 0.
\]
Thus it suffices to check 
\[
m = u \rho^\varphi \quad \mbox{whenever} \quad \rho > 0.
\]
For this, we introduce a set
\[
E_R^\delta := \lt\{ (x,t) \in (-R,R) \times (0, T)\,:\, \rho(x,t) > \delta \rt\}.
\]
Due to the compactness of $\rho_\e$ and $\rho_\e \star \theta_\e$, by Egorov's theorem, for any $\eta> 0$, there exists a set $C_\eta \subset E_R^\delta$ with $|E_R^\delta \setminus C_\eta | < \eta$ on which $\rho_\e$ and $\rho_\e \star \theta_\e$ uniformly converge to $\rho$. This asserts $\rho_\e \star \theta_\e > \delta/2$ in $C_\eta$ for $\e>0$ small enough. Thus we obtain
\[
u^\e_\e \rho^\varphi_\e = \frac{(\rho_\e u_\e)\star\theta_\e}{\e + \rho_\e \star \theta_\e} \rho^\varphi_\e \to m = u \rho^\varphi \quad \mbox{in} \quad C_\eta.
\]
This further yields
\[
m = u\rho^\varphi \quad \mbox{on} \quad \{ \rho > 0\},
\]
since $\eta > 0$, $R>0$, and $\delta > 0$ were arbitrary. Consequently, we have
\[
\int_0^{T} \int_{\R \times \R_+} f_\e u^\e_\e \psi\,dxdvdt \to \int_0^{T} \int_{\R} u \rho^\varphi \phi dxdt = \int_0^{T} \int_{\R \times \R_+} f u \psi\,dxdvdt
\]
for all test functions of the form $\psi(x,v,t) = \phi(x,t)\varphi(v)$. 

%$$\begin{aligned}
%\int_E |\rho_\e u_\e|\,dxdt &\leq \lt( \int_E \rho_\e|u_\e|^2\,dxdt \rt)^{1/2} \lt(\int_E \rho_\e\,dxdt  \rt)^{1/2}\cr
%&\leq \lt( \int_{E \times \R_+} f_\e|v|^2\,dxdvdt \rt)^{1/2} \lt(\int_E \rho_\e\,dxdt  \rt)^{1/2}\cr
%&\leq C\lt(\int_E \rho_\e\,dxdt  \rt)^{1/2} \to 0 \quad \mbox{as} \quad \e \to 0.
%\end{aligned}$$
%Thus we obtain $\rho u = 0$ a.e. in $E$.

%%%%%%%%%%%%%%%%%%%%%%%%%%%%%%%%%%%%%%%%%%%%%%%%%%%%%%%%%%%%%%%%%%%%%%%%%%%%%%%%%%%%%%%%%%%%%%%%%%%%%%%%%%%%%%%%%%%%%%%%%%%%%%%%%%%%%%%%%%%%%

\subsection{$(\rho_\e u_\e f_\e)/(1+\e(\rho_\e+\rho_\e u_\e))$ converges to $\rho u f$ in $L^\infty(0,T; L^p(\R \times \R_+))$}
Employing the same notations with that in the previous section, we get
\[
\int_0^{T} \int_{\R \times \R_+} \frac{\rho_\e u_\e f_\e}{1+\e(\rho_\e+\rho_\e u_\e)} \psi\,dxdvdt = \int_0^{T} \int_{\R} \frac{\rho_\e u_\e \rho^\varphi_\e }{1+\e(\rho_\e+\rho_\e u_\e)}\phi \,dxdt.
\]
Note that if $p \in (1, 3/2)$, then $3p/(2-p) \in (3, 9)$, and this gives the following uniform boundedness in $\e$:
\[
\lt\|\frac{\rho_\e u_\e \rho^\varphi_\e}{1+\e(\rho_\e+\rho_\e u_\e)}\rt\|_{L^p} \leq \|\rho_\e u_\e \rho^\varphi_\e\|_{L^p}\leq \|\varphi\|_{L^\infty}\|\rho_\e\|^{3/2}_{L^{3p/(2-p)}} \|\sqrt{\rho_\e} u_\e\|_{L^2} < \infty.
\]
This again gives the existence of a function $m \in L^\infty(0,T; L^p(\R))$ such that
\[
\frac{\rho_\e u_\e \rho^\varphi_\e}{1+\e(\rho_\e+\rho_\e u_\e)} \rightharpoonup m \quad \mbox{in} \quad L^\infty(0,T; L^p(\R)) \quad \mbox{for all} \quad p \in \lt(1, 3/2 \rt)
\]
up to a subsequence. We then show that
\[
m = \rho u \rho^\varphi, \quad \mbox{where} \quad \rho^\varphi = \int_{\R_+} f \varphi\,dv \mbox{ and } \rho u = \int_{\R_+} v f\,dv.
\]
By using the set $E_R^0= \lt\{ (x,t) \in (-R,R) \times (0, T)\,:\, \rho(x,t) =0 \rt\}$ again, we obtain
\[
\|u_\e \rho^\varphi_\e\|_{L^p(E_R^0)} \leq C \|\rho_\e\|^{3/2}_{L^{3p/(2-p)}(E_R^0)} \to 0 \quad \mbox{as} \quad \e \to 0.
\]
Again by considering the set $E^\delta_\R$ defined as before and employing almost the same argument as in the previous section, we can show that 
\[
m = \rho u \rho^\varphi \quad \mbox{whenever} \quad \rho > 0.
\]
%For this, we introduce a set
%\[
%E_R^\delta := \lt\{ (x,t) \in (-R,R) \times (0, T)\,:\, \rho(x,t) > \delta \rt\}.
%\]
%Due to the compactness of $\rho_\e$, $\rho_\e u_\e$, by Egorov's theorem, for any $\eta> 0$, there exists a set $C_\eta \subset E_R^\delta$ with $|E_R^\delta \setminus C_\eta | < \eta$ on which $\rho_\e u_\e$  uniformly converge to $\rho u$. %This asserts $\rho_\e \star \theta_\e > \delta/2$ in $C_\eta$ for $\e>0$ small enough. Thus we obtain
%\[
%\frac{\rho_e u_\e \rho^\varphi_\e}{1+\e(\rho_\e+\rho_\e u_\e)}   \to  \rho u \rho^\varphi \quad \mbox{in} \quad C_\eta.
%\]
%This further yields
%\[
%m = \rho u\rho^\varphi \quad \mbox{on} \quad \{ \rho > 0\},
%\]
%since $\eta > 0$, $R>0$, and $\delta > 0$ were arbitrary. 
Hence we have
\[
\int_0^{T} \int_{\R \times \R_+} \frac{\rho_\e u_\e f_\e}{1+\e(\rho_\e+\rho_\e u_\e)} \psi\,dxdvdt \to \int_0^{T} \int_{\R} u \rho^\varphi \phi dxdt = \int_0^{T} \int_{\R \times \R_+} f u \psi\,dxdvdt
\]
for all test functions of the form $\psi(x,v,t) = \phi(x,t)\varphi(v)$. 

%$$\begin{aligned}
%\int_E |\rho_\e u_\e|\,dxdt &\leq \lt( \int_E \rho_\e|u_\e|^2\,dxdt \rt)^{1/2} \lt(\int_E \rho_\e\,dxdt  \rt)^{1/2}\cr
%&\leq \lt( \int_{E \times \R_+} f_\e|v|^2\,dxdvdt \rt)^{1/2} \lt(\int_E \rho_\e\,dxdt  \rt)^{1/2}\cr
%&\leq C\lt(\int_E \rho_\e\,dxdt  \rt)^{1/2} \to 0 \quad \mbox{as} \quad \e \to 0.
%\end{aligned}$$
%Thus we obtain $\rho u = 0$ a.e. in $E$.

\begin{proof}[Proof of Theorem \ref{thm_main}] Equipped with the previous convergence estimates, there is no problem with passing to the limit $\e \to 0$ in \eqref{sys_reg} to conclude that the limiting function $f$ is a weak solution to \eqref{main_sys} in the sense of Definition \ref{def_weak}.
\end{proof}

\section{Hydrodynamic limit: Proof of Theorem \ref{thm_hydro}}\label{sec_hydro}
In this section, we present the details of proof of Theorem \ref{thm_hydro}. As mentioned before, our proof relies on the relative entropy argument.
\subsection{Relative entropy estimate}
We first rewrite the equations \eqref{Euler_eq} as a conservative form:
\[
U_t + \nabla \cdot A(U) = 0,
\]
where 
\[
m = \rho u, \quad U := \begin{pmatrix}
\rho \\
m 
\end{pmatrix},
\quad
A(U) := \begin{pmatrix}
m  \\
m^2/\rho
\end{pmatrix}.
\]
Then the above system have the macro entropy form $E(U) := m^2/(2\rho)$. Note that the entropy defined above is not strictly convex with respect to $\rho$. We now define the relative entropy functional $\mh$ as follows.
\bq\label{def_rel}
\mh(\bar U|U) := E(\bar U) - E(U) - DE(U)(\bar U-U) \quad \mbox{with} \quad \bar U := \begin{pmatrix}
        \bar\rho \\
        \bar m \\
    \end{pmatrix},
\eq
where $D E(U)$ denotes the derivation of $E$ with respect to $\rho, m$, i.e.,
\[
DE(U) = \begin{pmatrix}
\displaystyle        -m^2/(2\rho^2) \\
\displaystyle        m/\rho
    \end{pmatrix}.
\]
This asserts
\[
\mh(\bar U|U) = \frac{\bar\rho \bar u^2}{2} - \frac{\rho u^2}{2} - \frac{u^2}{2}(\rho - \bar \rho) - u\cdot (\bar \rho \bar u - \rho u)= \frac{\bar\rho}{2}(u - \bar u)^2.
\] 
As mentioned in Introduction, the relative entropy defined above does not give any information about the discrepancy between $\rho$ and $\bar \rho$. 
\begin{lemma}\label{lem_rel}The relative entropy $\mh$ defined in \eqref{def_rel} satisfies the following equality.
\begin{align*}
&\frac{d}{dt}\int_\R \mh(\bar U|U)\,dx \cr
&\quad= \int_\R \partial_t E(\bar U)\,dx - \int_\R \nabla (DE(U)):A(\bar U|U)\,dx - \int_\R DE(U)\lt( \pa_t \bar U + \nabla \cdot A(\bar U)\rt)dx,
\end{align*}
where $A(\bar U|U)$ is the relative flux functional given by 
\[
A(\bar U|U) := A(\bar U) - A(U) -DA(U)(\bar U-U).
\]
\end{lemma}
\begin{proof}It follows from \eqref{def_rel} that
\begin{align*}
\frac{d}{dt}\int_\R \mh(\bar U|U)\,dx & = \int_\R \partial_t E(\bar U)\,dx - \int_\R DE(U)(\pa_t \bar U + \nabla \cdot A(\bar U))\,dx \cr
&\quad +\int_\R D^2 E(U) \nabla \cdot A(U)(\bar U-U) + DE(U) \nabla \cdot A(\bar U)\,dx\cr
&=: \sum_{i=1}^3 I_i,
\end{align*}
where $I_3$ can be easily estimated as
\begin{align*}
I_3 &= \int_\R \left(\nabla DE(U)\right):\left( DA(U)(\bar U-U) - A(\bar U)\right)dx \cr
&= -\int_\R \left(\nabla DE(U)\right):\left(A(\bar U|U) + A(U)\right)dx\cr
&=-\int_\R \left(\nabla DE(U)\right):A(\bar U|U)\,dx.
\end{align*}
Here we used the fact from \cite{KMT15} that
\[
\int_\R \left(\nabla DE(U)\right):A(U)\,dx =0.
\]
\end{proof}
We now set 
\[
m^\e := \rho^\e u^\e \quad \mbox{and} \quad U^\e := \begin{pmatrix} \rho^\e \\ m^\e  \end{pmatrix} \quad \mbox{with} \quad \rho^\e := \intr f^\e\,dv \quad \mbox{and} \quad  m^\e := \intr v f^\e\,dv,
\]
where $f^\e$ is a weak solution to the equation \eqref{main_hydro} in the sense of Definition \ref{def_weak}.
\begin{proposition} Let $f^\e$ be a weak solution to the equation \eqref{main_hydro} and $(\rho,u)$ be a strong solution to the system \eqref{Euler_eq} on the time interval $[0,T]$. Suppose that the assumptions {\bf (H1)}--{\bf (H2)} hold.  Then we have
\begin{align}\label{est_h1}
&\sup_{0 \leq t \leq T}\int_\R \mh(U^\e(t)|U(t))\,dx  \leq \mathcal{O}(\e).
\end{align}
\end{proposition}
\begin{proof}By the relative entropy estimate in Lemma \ref{lem_rel}, we obtain
\begin{align*}
&\int_\R \mh(U^\e|U)\,dx\cr
&\quad = \int_\R \mh(U_0^\e|U_0)\,dx + \int_\R (E(U^\e) - E(U^\e_0))\,dx - \int_0^t \int_\R \nabla (DE(U)):A(U^\e|U)\,dx ds \cr
&\qquad - \int_0^t \int_\R DE(U)\lt( \pa_s U^\e + \nabla \cdot A(U^\e)\rt) \,dx ds \cr
&\quad =: \sum_{i=1}^4 J_i^\e.
\end{align*}
The assumption {\bf (H1)} gives $J_1^\e = \mathcal{O}(\e)$. For the estimate of $J_2^\e$, we first notice that 
\[
(u^\e)^2 =  \lt(\frac{\intr v f^\e\,dv}{\intr f^\e\,dv} \rt)^2 \leq \frac{ \intr v^2 f^\e\,dv}{\rho^\e}, \quad \mbox{i.e.,} \quad \rho^\e(u^\e)^2 \leq \intr v^2 f^\e\,dv,
\]
and this implies
\bq\label{est_j20}
E(U^\e) = \frac12\int_\R \rho^\e(u^\e)^2\,dx  \leq \frac12\intr v^2 f^\e\,dx dv.
\eq
Thus, by adding and subtracting, we find
\begin{align}\label{est_j2}
\begin{aligned}
J_2^\e &= \int_\R E(U^\e)\,dx - \intrr v^2 f^\e\,dxdv \cr
&\quad + \intrr v^2 f^\e\,dxdv - \intrr v^2 f^\e_0\,dxdv \cr
&\quad + \intrr v^2 f^\e_0\,dxdv - \int_\R E(U_0)\,dx\cr
&\leq 0 + 0 + \mathcal{O}(\e).
\end{aligned}
\end{align}
Here we also used Lemma \ref{lem_lp} and {\bf (H1)}. 
We next use \cite[Proposition 2.2]{CC_pre} to estimate 
\[
J_3^\e \leq C\int_0^t \int_\R \mh(U^\e|U)\,dx ds,
\]
by using the following identity:
\[
\int_\R |A(U^\e|U)|\,dx = \int_\R \rho^\e(u^\e-u)^2\,dx.
\]
Moreover, by Theorem \ref{thm_main}, see also Lemma \ref{lem_lp}, we get that $f^\e$ satisfies
\bq\label{gd_est}
\int_0^t \int_{\R \times \R_+} f^\e (u^\e -v)^2\,dxdvds \leq \frac\e2\int_{\R \times \R_+} v^2 f^\e_0\,dxdv,
\eq
and this together with the assumption $({\bf H1})$ asserts 
\begin{align*}
J_4^\e &\leq \|\pa_x u\|_{L^\infty} \int_0^t \int_\R \lt|\intr ((u^\e)^2 - v^2)f^\e\,dv \rt|dxds\cr
&\leq \|\pa_x u\|_{L^\infty} \int_0^t \int_{\R \times \R_+} f^\e (u^\e -v)^2\,dxdvds\cr
&\leq C\e, 
\end{align*}
where $C = C(\|\pa_x u\|_{L^\infty}, \int_{\R \times \R_+} f^\e_0 v^2\,dxdv) > 0$. We finally combine all of the above estimates to conclude the proof.
\end{proof}
\subsection{Proof of Theorem \ref{thm_hydro}} We now show that the MKR distance can be bounded by the relative entropy, which directly gives the quantitative error estimate between $\rho$ and $\rho^\e$.

Note that the local densities $\rho$ and $\rho^\e$ satisfy
\[
\pa_t \rho + \pa_x(\rho u) = 0 \quad \mbox{and}  \quad \pa_t \rho^\e + \pa_x(\rho^\e u^\e) = 0,
\]
respectively. Let us define forward characteristics $X(t) := X(t;0,x)$ and $X^\e(t) := X^\e(t;0,x)$, $t \in [0,T]$ as solutions to 
\bq\label{eq_char}
\pa_t X(t) = u(X(t),t) \quad \mbox{and}  \quad \pa_t X^\e(t) = u^\e(X^\e(t),t)
\eq
with $X(0) = X^\e(0) = x \in \R$, respectively. Since $u$ is bounded and Lipschitz continuous on the time interval $[0,T]$, we find the solution $\rho$ uniquely exists and it can be determined as  the push-forward of the its initial densities through the flow maps $X$, i.e.,  $\rho(t) = X(t;0,\cdot) \# \rho_0$. Here $\cdot \,\# \,\cdot $ stands for the push-forward of a probability measure by a measurable map, more precisely, $\nu = \mt \# \mu$ for probability measure $\mu$ and measurable map $\mt$ implies
\[
\int_\R \varphi(y) \,d\nu(y) = \int_\R \varphi(\mt(x)) \,d\mu(x),
\]
for all $\varphi \in \mc_b(\R)$. On the other hand, due to the lack of regularity of $u^\e$, it is not clear to have the existence of solutions $X^\e$. To handle this issue, we recall the following proposition from \cite[Theorem 8.2.1]{AGS08}, see also \cite[Proposition 3.3]{FK19}. 
\begin{proposition}\label{prop_am}Let $T>0$ and $\rho : [0,T] \to \mathcal{P}(\R)$ be a narrowly continuous solution of \eqref{eq_char}, that is, $\rho$ is continuous in the duality with continuous bounded functions, for a Borel vector field $u$ satisfying
\bq\label{est_p1}
\int_0^T\int_\R |u(x,t)|^p\rho(x,t)\,dx dt < \infty,
\eq
for some $p > 1$. Let $\Gamma_T: [0,T] \to \R$ denote the space of continuous curves. Then there exists a probability measure $\eta$ on $\Gamma_T \times \R$ satisfying the following properties:
\begin{itemize}
\item[(i)] $\eta$ is concentrated on the set of pairs $(\gamma,x)$ such that $\gamma$ is an absolutely continuous curve satisfying
\[
\dot\gamma(t) = u(\gamma(t),t), 
\]
for almost everywhere $t \in (0,T)$ with $\gamma(0) = x \in \R$.
\item[(ii)] $\rho$ satisfies
\[
\int_\R \varphi(x)\rho\,dx = \int_{\Gamma_T \times \R}\varphi(\gamma(t))\,d\eta(\gamma,x),
\]
for all $\varphi \in \mc_b(\R)$, $t \in [0,T]$.
\end{itemize}
\end{proposition}
Note that it follows from \eqref{est_j20} and \eqref{est_j2} that
\[
\int_\R (u^\e)^2 \rho^\e\,dx \leq \int_{\R \times \R_+} v^2 f^\e\,dxdv \leq \int_{\R \times \R_+} v^2 f^\e_0\,dxdv < \infty,
\]
i.e., \eqref{est_p1} holds for $p=2$, and thus by Proposition \ref{prop_am}, we have the existence of a probability measure $\eta^\e$ in $\Gamma_T \times \R$, which is concentrated on the set of pairs $(\gamma,x)$ such that $\gamma$ is a solution of
\bq\label{eq_gam}
\dot{\gamma}(t) = u^\e(\gamma(t),t),
\eq
with $\gamma(0) = x$. Furthermore, it holds
\bq\label{eq_gam2}
\int_\R \varphi(x) \rho^\e(x,t)\,dx = \int_{\Gamma_T \times \R}\varphi(\gamma(t))\,d\eta^\e(\gamma,x),
\eq
for all $\varphi \in \mc_b(\R)$, $t \in [0,T]$.

% and subsequently, by the disintegration theorem of measures \cite[Theorem 5.3.1]{AGS08}, we have $
%d\eta^\e(\gamma,x) = \eta^\e_x(d\gamma)\otimes \rho^\e_0(x)\,dx$, where $\{\eta^\e_x\}$ is a family of probability measures on $\Gamma_T$ concentrated on solutions of \eqref{eq_gam}. 

We now consider the push-forward of the $\rho_0^\e$ through the flow map $X$ and denote it by $\bar\rho^\e$, i.e., $\bar\rho^\e = X \# \rho_0^\e$. Then for bounded Lipschitz function $\phi$, we find
\bq\label{est_rho1}
\lt|\int_\R \phi(x) (\rho(x) - \bar\rho^\e(x))\,dx\rt| = \lt|\int_\R \phi(X(t))(\rho_0(x) - \rho_0^\e(x))\,dx  \rt| \leq Cd_{MKR}(\rho_0, \rho_0^\e),
\eq
where $C>0$ is independent of $\e$, and we used the fact that the bounded Lipschitz distance is equivalent to MKR distance and $\phi(X)$ is bounded and Lipschitz. Indeed, we get 
\begin{align*}
|X(t;0,x) - X(t;0,y)| &\leq |x-y| + \int_0^t |u(X(s;0,x)) - |u(X(s;0,y))|\,ds\cr
&\leq |x-y| + \|\pa_x u\|_{L^\infty}\int_0^t |X(s;0,x) - X(s;0,y)|\,ds
\end{align*}
and  apply Gr\"onwall's lemma to derive the Lipschitz continuity of the characteristic flow $X(t;0,x)$ in $x$, and subsequently, this asserts
\[
|\phi(X(t;0,x)) - \phi(X(t;0,y))| \leq \|\phi\|_{Lip}|X(t;0,x) - X(t;0,y)| \leq \|\phi\|_{Lip} \|X\|_{Lip}|x-y|,
\]
where $\|\cdot\|_{Lip}$ denotes the Lipschitz constant given by
\[
\|\phi\|_{Lip} := \sup_{x \neq y \in \R} \frac{|\phi(x) - \phi(y)|}{|x-y|}.
\]
Thus we obtain from \eqref{est_rho1} that
\bq\label{est_rho11}
d_{MKR}(\rho(t), \bar\rho^\e(t)) \leq Cd_{MKR}(\rho_0, \rho_0^\e),
\eq
for $t \in [0,T]$, where $C > 0$ is independent of $\e > 0$. We next estimate the error between $\bar\rho^\e$ and $\rho^\e$. For this, we note that by the disintegration theorem of measures, see \cite{AGS08}, we can write
\[
d\eta^\e(\gamma,x) = \eta^\e_x(d\gamma) \otimes \rho^\e_0(x)\,dx,
\]
where $\{\eta^\e_x\}_{x \in \R}$ is a family of probability measures on $\Gamma_T$ concentrated on solutions of \eqref{eq_gam}. We then introduce a measure $\nu^\e$ on $\Gamma_T \times \Gamma_T \times \R$ defined by
\[
d\nu^\e(\gamma, x, \sigma) = \eta^\e_x(d\gamma) \otimes \delta_{X(\cdot;0,x)}(d\sigma) \otimes \rho^\e_0(x)\,dx.
\]
We further consider an evaluation map $E_t : \Gamma_T  \times \Gamma_T \times \R \to \R \times \R$ defined as $E_t(\gamma, \sigma, x) = (\gamma(t), \sigma(t))$. Then we readily find that measure $\pi^\e_t:= (E_t)\# \nu^\e$ on $\R \times \R$ has marginals $\rho^\e(x,t)\,dx$ and $\bar\rho^\e(y,t)\,dy$ for $t \in [0,T]$, see \eqref{eq_gam2}. This yields
\begin{align}\label{est_rho2}
\begin{aligned}
d_{MKR}(\rho^\e(t), \bar\rho^\e(t)) &\leq \int_{\R \times \R} |x-y|\,d\pi^\e_t(x,y)\cr
&=\int_{\Gamma_T \times \Gamma_T \times \R} |\sigma(t) - \gamma(t) | \,d\nu^\e(\gamma, \sigma, x) \cr
&= \int_{\Gamma_T \times \R} |X(t;0,x) - \gamma(t)| \,d\eta^\e(\gamma,x).
\end{aligned}
\end{align}
On the other hand, it follows from \eqref{eq_char} and \eqref{eq_gam} that
\begin{align*}
&\lt|X(t;0,x) -\gamma(t)\rt| \cr
&\quad = \lt|\int_0^t u(X(s;0,x)) - u^\e(\gamma(s),s)\,ds\rt|\cr
&\qquad \leq \int_0^t \lt|u(X(s;0,x)) - u(\gamma(s),s)\rt|ds + \int_0^t \lt|u(\gamma(s),s) - u^\e(\gamma(s),s)\rt|ds\cr
&\qquad \leq \|\pa_x u\|_{L^\infty}\int_0^t  \lt|X(s;0,x) - \gamma(s)\rt|ds + \int_0^t \lt|u(\gamma(s),s) - u^\e(\gamma(s),s)\rt|ds.
\end{align*}
Applying Gr\"onwall's lemma to the above asserts
\[
\lt|X(t;0,x) -\gamma(t)\rt| \leq C\int_0^t \lt|u(\gamma(s),s) - u^\e(\gamma(s),s)\rt|ds,
\]
where $C>0$ is independent of $\e>0$. Combining this with \eqref{est_rho2}, we have
\begin{align}\label{est_rho22}
\begin{aligned}
d_{MKR}(\rho^\e(t), \bar\rho^\e(t)) &\leq C\int_0^t \int_{\Gamma_T \times \R} \lt|u(\gamma(s),s) - u^\e(\gamma(s),s)\rt|d\eta^\e(\gamma,x)\,ds\cr
&\leq C\int_0^t \int_\R |u(x,s) - u^\e(x,s)| \rho^\e(x,s)\,dxds\cr
&\leq C\sqrt T\lt( \int_0^t \int_\R (u^\e(x,s) - u(x,s))^2 \rho^\e(x,s) \,dx ds\rt)^{1/2}\cr
&=C\lt(\int_0^t  \int_\R \mh(U^\e|U)\,dx ds\rt)^{1/2},
\end{aligned}
\end{align}
where $C>0$ is independent of $\e > 0$, and we used \eqref{eq_gam2}. We then combine \eqref{est_rho11} and \eqref{est_rho22} to conclude
\begin{align*}
d_{MKR}(\rho(t), \rho^\e(t)) &\leq d_{MKR}(\rho(t), \bar\rho^\e(t)) + d_{MKR}(\rho^\e(t), \bar\rho^\e(t))\cr
&\leq Cd_{MKR}(\rho_0, \rho_0^\e) + C\lt(\int_0^t  \int_\R \mh(U^\e|U)\,dx ds\rt)^{1/2},
\end{align*}
where $C>0$ is independent of $\e > 0$. We finally use the estimate \eqref{est_h1} to conclude that 
\[
d_{MKR}(\rho(t), \rho^\e(t)) \leq Cd_{MKR}(\rho_0, \rho_0^\e) + \mathcal{O}(\sqrt\e).
\]
We next estimate the MKR distance between $f^\e$ and $\rho\otimes \delta_u$. For $\phi \in Lip(\R \times \R_+)$, we estimate
\begin{align*}
&\lt| \int_{\R \times \R_+} \phi(x,v) f^\e(x,v)\,dxdv - \int_{\R \times \R_+} \phi(x,v)\rho(x)\,dx \otimes \delta_{u(x)}(dv)\rt|\cr
&\quad = \lt| \int_{\R \times \R_+} \phi(x,v) f^\e(x,v)\,dxdv - \int_\R \phi(x,u(x))\rho(x)\,dx\rt|\cr
&\quad \leq \lt| \int_{\R \times \R_+} \phi(x,v) f^\e(x,v)\,dxdv - \int_\R \phi(x,u(x))\rho^\e(x)\,dx\rt|\cr
&\qquad + \lt|\int_\R \phi(x,u(x))\rho^\e(x)\,dx - \int_\R \phi(x,u(x))\rho(x)\,dx\rt|.
\end{align*}
Since both $\phi$ and $u$ are Lipschitz, the second term on the right hand side of the above inequality can be bounded by 
\[
Cd_{MKR}(\rho^\e,\rho) \leq \mathcal{O}(\sqrt \e),
\]
where $C > 0$ is independent of $\e$. For the first term, we estimate
\begin{align*}
&\lt| \int_{\R \times \R_+} \phi(x,v) f^\e(x,v)\,dxdv - \int_\R \phi(x,u(x))\rho^\e(x)\,dx\rt|\cr
&\quad \leq \|\phi\|_{Lip} \int_{\R \times \R_+} |v - u| f^\e(x,v)\,dxdv\cr
&\quad \leq \|\phi\|_{Lip} \lt(\int_{\R \times \R_+} |v - u^\e| f^\e(x,v)\,dxdv + \int_\R |u^\e - u| \rho^\e(x)\,dx \rt)\cr
&\quad \leq C\lt( \int_{\R \times \R_+} (v - u^\e)^2 f^\e(x,v)\,dxdv + \int_\R (u^\e - u)^2 \rho^\e(x)\,dx \rt)^{1/2}\cr
&\quad \leq C\sqrt\e,
\end{align*}
due to \eqref{est_h1} and \eqref{gd_est}, where $C > 0$ is independent of $\e$. Combining all of the above estimates, we have
\[
\lt| \int_{\R \times \R_+} \phi(x,v) f^\e(x,v)\,dxdv - \int_{\R \times \R_+} \phi(x,v)\rho(x)\,dx \otimes \delta_{u(x)}(dv)\rt| \leq \mathcal{O}(\sqrt\e)
\]
for any $\phi \in Lip(\R \times \R_+)$. This completes the proof.

%%%%%%%%%%%%%%%%%%%%%%%%%%%%%%%%%%%%%%%%%%%%%%%%%%%%%%%%%%%%%%%%%%%%%%%%%%%%%%%%%%%%%%%%%%%%%%%%%%%%%%%%%%%%%%%%%%%%%%%%%%%%%%%%%%%%%%%%%%%%%%
%
%
%    Acknowledgments
%
%
%%%%%%%%%%%%%%%%%%%%%%%%%%%%%%%%%%%%%%%%%%%%%%%%%%%%%%%%%%%%%%%%%%%%%%%%%%%%%%%%%%%%%%%%%%%%%%%%%%%%%%%%%%%%%%%%%%%%%%%%%%%%%%%%%%%%%%%%%%%%%%%%

\section*{Acknowledgments}
Young-Pil Choi is supported by National Research Foundation of Korea(NRF) grants funded by the Korea government(MSIP) (No. 2017R1C1B2012918) and POSCO Science Fellowship of POSCO TJ Park Foundation. Seok-Bae Yun is supported by Samsung Science and Technology Foundation under Project Number SSTF-BA1801-02.

%%%%%%%%%%%%%%%%%%%%%%%%%%%%%%%%%%%%%%%%%%%%%
%
%
%        thebibliography
%
%
%%%%%%%%%%%%%%%%%%%%%%%%%%%%%%%%%%%%%%%%%%%%%

\end{document}